\documentclass[12pt]{article}

\usepackage{calrsfs}

\usepackage{ amsthm,amsfonts,amsbsy,amssymb,amsmath,amscd}
\usepackage{amsmath}
\usepackage[all]{xy}
\usepackage{xr-hyper}
\usepackage[colorlinks=true]{hyperref}
\usepackage[nottoc]{tocbibind}

\usepackage{fullpage}
\usepackage{verbatim}
\usepackage{calrsfs}
\usepackage[dvips]{lscape}

\usepackage{tikz-cd}

\usepackage{hyperref} 

\begin{document}
\title{
Standard Conjectures and Height Pairings}
\author{Shou-Wu Zhang}
\maketitle

\setcounter{section}{-1}

\theoremstyle{plain}
\newtheorem{thm}{Theorem}[subsection]
\newtheorem{theorem}[thm]{Theorem}
\newtheorem{proposition}[thm]{Proposition}
\newtheorem{prop}[thm]{Proposition}
\newtheorem{corollary}[thm]{Corollary}
\newtheorem{cor}[thm]{Corollary}
\newtheorem{lemma}[thm]{Lemma}
\newtheorem{lem}[thm]{Lemma}
\newtheorem{conjecture}[thm]{Conjecture}
\newtheorem{conj}[thm]{Conjecture}

\theoremstyle{Definition}
\newtheorem{definition}[thm]{Definition}
\newtheorem{defn}[thm]{Definition}

\theoremstyle{remark} 
\newtheorem{remark}[thm]{Remark}
\newtheorem{example}[thm]{Example}
\newtheorem{notation}[thm]{Notation}
\newtheorem{problem}[thm]{Problem}

\numberwithin{equation}{subsection}

 \newcommand{\sech}{ \mathrm{sech}\,}
 \newcommand{\csch}{\mathrm{csch}\,}

\newcommand{\BA}{{\mathbb {A}}}
\newcommand{\BB}{{\mathbb {B}}}
\newcommand{\BC}{{\mathbb {C}}}
\newcommand{\BD}{{\mathbb {D}}}
\newcommand{\BE}{{\mathbb {E}}}
\newcommand{\BF}{{\mathbb {F}}}
\newcommand{\BG}{{\mathbb {G}}}
\newcommand{\BH}{{\mathbb {H}}}
\newcommand{\BI}{{\mathbb {I}}}
\newcommand{\BJ}{{\mathbb {J}}}
\newcommand{\BK}{{\mathbb {K}}}
\newcommand{\BL}{{\mathbb {L}}}
\newcommand{\BM}{{\mathbb {M}}}
\newcommand{\BN}{{\mathbb {N}}}
\newcommand{\BO}{{\mathbb {O}}}
\newcommand{\BP}{{\mathbb {P}}}
\newcommand{\BQ}{{\mathbb {Q}}}
\newcommand{\BR}{{\mathbb {R}}}
\newcommand{\BS}{{\mathbb {S}}}
\newcommand{\BT}{{\mathbb {T}}}
\newcommand{\BU}{{\mathbb {U}}}
\newcommand{\BV}{{\mathbb {V}}}
\newcommand{\BW}{{\mathbb {W}}}
\newcommand{\BX}{{\mathbb {X}}}
\newcommand{\BY}{{\mathbb {Y}}}
\newcommand{\BZ}{{\mathbb {Z}}}

\newcommand{\CA}{{\mathcal {A}}}
\newcommand{\CB}{{\mathcal {B}}}
\newcommand{\CC}{{\mathcal{C}}}
\renewcommand{\CD}{{\mathcal{D}}}
\newcommand{\CE}{{\mathcal {E}}}
\newcommand{\CF}{{\mathcal {F}}}
\newcommand{\CG}{{\mathcal {G}}}
\newcommand{\CH}{{\mathcal {H}}}
\newcommand{\CI}{{\mathcal {I}}}
\newcommand{\CJ}{{\mathcal {J}}}
\newcommand{\CK}{{\mathcal {K}}}
\newcommand{\CL}{{\mathcal {L}}}
\newcommand{\CM}{{\mathcal {M}}}
\newcommand{\CN}{{\mathcal {N}}}
\newcommand{\CO}{{\mathcal {O}}}
\newcommand{\CP}{{\mathcal {P}}}
\newcommand{\CQ}{{\mathcal {Q}}}
\newcommand{\CR }{{\mathcal {R}}}
\newcommand{\CS}{{\mathcal {S}}}
\newcommand{\CT}{{\mathcal {T}}}
\newcommand{\CU}{{\mathcal {U}}}
\newcommand{\CV}{{\mathcal {V}}}
\newcommand{\CW}{{\mathcal {W}}}
\newcommand{\CX}{{\mathcal {X}}}
\newcommand{\CY}{{\mathcal {Y}}}
\newcommand{\CZ}{{\mathcal {Z}}}

\newcommand{\RA}{{\mathrm {A}}}
\newcommand{\RB}{{\mathrm {B}}}
\newcommand{\RC}{{\mathrm {C}}}
\newcommand{\RD}{{\mathrm {D}}}
\newcommand{\RE}{{\mathrm {E}}}
\newcommand{\RF}{{\mathrm {F}}}
\newcommand{\RG}{{\mathrm {G}}}
\newcommand{\RH}{{\mathrm {H}}}
\newcommand{\RI}{{\mathrm {I}}}
\newcommand{\RJ}{{\mathrm {J}}}
\newcommand{\RK}{{\mathrm {K}}}
\newcommand{\RL}{{\mathrm {L}}}
\newcommand{\RM}{{\mathrm {M}}}
\newcommand{\RN}{{\mathrm {N}}}
\newcommand{\RO}{{\mathrm {O}}}
\newcommand{\RP}{{\mathrm {P}}}
\newcommand{\RQ}{{\mathrm {Q}}}
\newcommand{\RS}{{\mathrm {S}}}
\newcommand{\RT}{{\mathrm {T}}}
\newcommand{\RU}{{\mathrm {U}}}
\newcommand{\RV}{{\mathrm {V}}}
\newcommand{\RW}{{\mathrm {W}}}
\newcommand{\RX}{{\mathrm {X}}}
\newcommand{\RY}{{\mathrm {Y}}}
\newcommand{\RZ}{{\mathrm {Z}}}

\newcommand{\fa}{{\frak a}}
\newcommand{\fg}{{\frak g}}
\newcommand{\fp}{{\frak p}}
\newcommand{\fk}{{\frak k}}
\newcommand{\fq}{{\frak q}}
\newcommand{\fl}{{\frak l}}
\newcommand{\fm}{{\frak m}}

\newcommand{\sL}{{\mathsf L}}
\newcommand{\sLambda}{{\mathsf \Lambda}}

\newcommand{\ab}{{\mathrm{ab}}}
\newcommand{\Ad}{{\mathrm{Ad}}}
\newcommand{\ad}{{\mathrm{ad}}}
\newcommand{\adm}{{\mathrm{adm}}}
\newcommand{\Alb }{{\mathrm{Alb}}}
\newcommand{\an}{{\mathrm{an}}}
\newcommand{\Aut}{{\mathrm{Aut}}}

\newcommand{\bata}{{\bar\eta}}
\newcommand{\Br}{{\mathrm{Br}}}
\newcommand{\bs}{\backslash}
\newcommand{\bbs}{\|\cdot\|}
\newcommand{\brs}{{\breve s}}

\newcommand{\Ch}{{\mathrm{Ch}}}
\newcommand{\cod}{{\mathrm{cod}}}
\newcommand{\coker}{{\mathrm{coker}}}
\newcommand{\Coker}{{\mathrm{Coker}}}
\newcommand{\cont}{{\mathrm{cont}}}
\newcommand{\cl}{{\mathrm{cl}}}
\newcommand{\closed}{{\mathrm{closed}}}
\newcommand{\cm}{{\mathrm{cm}}}

\newcommand{\der}{{\mathrm{der}}}
\newcommand{\dR}{{\mathrm{dR}}}
\newcommand{\disc}{{\mathrm{disc}}}
\newcommand{\Div}{{\mathrm{Div}}}
\renewcommand{\div}{{\mathrm{div}}}

\newcommand{\eff}{{\mathrm{eff}}}
\newcommand{\Eis}{{\mathrm{Eis}}}
\newcommand{\End}{{\mathrm{End}}}
\newcommand{\Ext}{{\mathrm{Ext}}}
\newcommand{\emb}{{\hookrightarrow}}

\newcommand{\fgl}{{\frak {gl}}}

\newcommand{\Frob}{{\mathrm{Frob}}}

\newcommand{\Gal}{{\mathrm{Gal}}}
\newcommand{\GL}{{\mathrm{GL}}}
\newcommand{\GO}{{\mathrm{GO}}}
\newcommand{\GSO}{{\mathrm{GSO}}}
\newcommand{\GSp}{{\mathrm{GSp}}}
\newcommand{\GSpin}{{\mathrm{GSpin}}}
\newcommand{\GU}{{\mathrm{GU}}}
\newcommand{\gr}{{\mathrm{Gr}}}

\newcommand{\Ara}{{\mathrm{Ara}}}
\newcommand{\HC}{{\mathrm{HC}}}
\newcommand{\Hom}{{\mathrm{Hom}}}
\newcommand{\Hol}{{\mathrm{Hol}}}

\renewcommand{\Im}{{\mathrm{Im}}}
\newcommand{\Ind}{{\mathrm{Ind}}}
\newcommand{\inv}{{\mathrm{inv}}}
\newcommand{\Isom}{{\mathrm{Isom}}}

\newcommand{\Jac}{{\mathrm{Jac}}}
\newcommand{\JL}{{\mathrm{JL}}}

\newcommand{\Ker}{{\mathrm{Ker}}}
\newcommand{\KS}{{\mathrm{KS}}}

\newcommand{\Lie}{{\mathrm{Lie}}}

\newcommand{\new}{{\mathrm{new}}}
\newcommand{\NS}{{\mathrm{NS}}}
\newcommand{\NT}{{\mathrm{NT}}}

\newcommand{\ord}{{\mathrm{ord}}}
\newcommand{\ol}{\overline}

\newcommand{\rank}{{\mathrm{rank}}}

\newcommand{\perv}{\mathrm{perv}}
\newcommand{\Perv}{\mathrm{Perv}}
\newcommand{\pH}{{^p\!\CH}}
\newcommand{\PGL}{{\mathrm{PGL}}}
\newcommand{\pR}{{^p\!R}}
\newcommand{\PSL}{{\mathrm{PSL}}}
\newcommand{\Pic}{\mathrm{Pic}}
\newcommand{\Prep}{\mathrm{Prep}}
\newcommand{\Proj}{\mathrm{Proj}}
\newcommand{\PU}{\mathrm{PU}}

\renewcommand{\Re}{{\mathrm{Re}}}
\newcommand{\red}{{\mathrm{red}}}
\newcommand{\reg}{{\mathrm{reg}}}
\newcommand{\Res}{{\mathrm{Res}}}
\newcommand{\RHom}{{\mathcal{RH}om}}

\newcommand{\Sel}{{\mathrm{Sel}}}
\newcommand{\Sh}{{\mathrm{Sh}}}
\newcommand{\Shc}{{\mathcal {Sh}}}
\font\cyr=wncyr10  \newcommand{\Sha}{\hbox{\cyr X}}
\newcommand{\SL}{{\mathrm{SL}}}
\renewcommand{\sl}{{\mathfrak{sl}}}
\newcommand{\SO}{{\mathrm{SO}}}
\newcommand{\Sp}{\mathrm{Sp}}
\newcommand{\Spec}{{\mathrm{Spec}}}
\newcommand{\Sym}{{\mathrm{Sym}}}
\newcommand{\sgn}{{\mathrm{sgn}}}
\newcommand{\Supp}{{\mathrm{Supp}}}

\newcommand{\TC}{{\mathrm{TC}}}

\newcommand{\tr}{{\mathrm{tr}}}

\newcommand{\ur}{{\mathrm{ur}}}

\newcommand{\vol}{{\mathrm{vol}}}
\newcommand{\zar}{{\mathrm{Zar}}}

\newcommand{\WH}{{^W\!H}}
\newcommand{\ws}{{\wt s}}
\newcommand{\wta}{{\wt\eta}}
\newcommand{\wt}{\widetilde}
\newcommand{\pp}{\frac{\partial\bar\partial}{\pi i}}
\newcommand{\ppr}{\frac{\partial\bar\partial}{2\pi i}}
\newcommand{\intn}[1]{\left( {#1} \right)}
\newcommand{\norm}[1]{\|{#1}\|}
\newcommand{\sfrac}[2]{\left( \frac {#1}{#2}\right)}
\newcommand{\ds}{\displaystyle}
\newcommand{\ov}{\overline}
\newcommand{\incl}{\hookrightarrow}
\newcommand{\imp}{\Longrightarrow}
\newcommand{\lto}{\longmapsto}
\newcommand{\surj}{\twoheadrightarrow}

\newcommand{\pair}[1]{\langle {#1} \rangle}

\newcommand{\prim}{\mathrm{prim}}
\newcommand{\wpair}[1]{\left\{{#1}\right\}}
\newcommand\wh{\widehat}
\newcommand\Spf{\mathrm{Spf}}
\newcommand{\lra}{{\longrightarrow}}
\newcommand{\iso}{{\overset\sim\lra}}

\newcommand{\Ei}{\mathrm{Ei}} 


\begin{abstract}
In this article, we extend Grothendieck's standard conjectures \cite[Conjectures 1, 2]{Gr} to cycles on degenerated fibers and use them to define some decompositions for the arithmetic  Chow group of Gillet--Soul\'e. In a local setting, our decompositions provide non-archimedean analogs of  ``{\em harmonic forms}" on K\"ahler manifolds. In a global setting, our decompositions provide canonical arithmetic liftings called ``{\em $\sL$- liftings}"  of algebraic cycles on varieties over number fields and thus provide a new height pairing called the {\em $\sL$-height pairing}  as one extension of Beilinson--Bloch's pairing of homologically trivial cycles.\end{abstract}
\medskip

\centerline{\em To Benedict Gross on his 70th birthday}

\tableofcontents

\section{Introduction}\label{sec-int}

Let $K$ be a number field with $\CO_K$ its ring of integers. Let $f: X\lra \Spec \CO_K$ be a regular arithmetic variety
with a polarization $L$. 
This means that $X$ is regular,  that  $f$ is flat and proper with geometrically connected fibers, and that $L$ is an ample hermitian line bundle in the sense of \cite{Zh}. 
In this paper, we introduce a new height pairing for the Chow group $\Ch^*(X_K)$.
This so-called $\sL$-pairing  extends Beilinson--Bloch's height paring on the subgroup 
$\Ch^*(X_K)^0$ of cycles homologous to $0$.
 Our pairing is conditional on some local and global standard conjectures.
Under our local standard conjectures, we introduce an Arakelov Chow group $\ol\Ch^*(X)$ of cycles 
whose curvature at each place of $K$ are harmonic forms. 
Under our global standard conjectures,  we construct  an $\sL$-lifting for the surjection $\ol\Ch^*(X)\to \Ch^*(X_K)$. 
Our work can be viewed as a preliminary step towards {\em a  Hodge theory for polarized arithmetic varieties}.

When $X$ is an arithmetic surface, our constructions in this paper are more or less well known.
First,  Arakelov \cite{Ar} introduced   
 a compactification  of $X$ by choosing some volume forms $\mu_v$ on the  Riemann surface $X_v(\BC)$ for each archimedean place $v$ of $K$.
 Then, he constructed  an intersection pairing  on the group $\ol\Pic (X)$ of Hermitian  line bundles 
 on $X$ with {\em admissible } metrics in  sense  that their 
 curvatures  on  $X_v(\BC)$ are  multiples of $\mu_v$ for each $v\mid \infty$.
 In \cite{Hr, Fa}, Hriljac and Faltings  independently proved a Hodge index theorem which provides 
 an intersection theoretical way to define the N\'eron--Tate heights on the Jacobian of $X_K$.
 Shortly after that, there were two developments in opposite directions. 
 First of all, Deligne \cite{De85} constructed an intersection pairing on the group $\wh \Pic (X)$ of all metrized line bundles without fixing $\mu_v$.
 Secondly, in a series of papers \cite{Ru, CR, Zh93, BGS95}, Rumely, Chinburg--Rumeley, Bloch--Gillet--Soul\'e and the author constructed some new intersection pairings on more restricted Arakelov Picard group $\ol\Pic (X, \mu)$ by choosing 
 metrics $\mu_v$ on the reduction graphs at {\em every bad places $v$}. For example,  $\mu_v$'s can be taken as the curvatures of an arithmetically ample line bundle $L$ as above. These treated archimedean and non-archimedean places more uniformly. One aim of this paper is to find a similar construction in high dimension case, i.e., {\em construction of harmonic forms, admissible arithmetic Chow cycles, and Hodge decompositions. }
 
 There were partial and delicate developments when $X$ is a high dimensional arithmetic variety.
 First of all,  for Arakelov's original theory, Beilinson and Gillet--Soul\'e independently in \cite{Be, GS84}  introduced compactifications of $X$ by choosing some K\"ahler forms  $\mu _v$  on complex manifold $X_v(\BC)$ for $v\mid\infty$.
 Then, they constructed intersection pairing on the   Arakelov Chow group $\ol \Ch^*(X, \mu)$ using currents with harmonic curvatures on each $X_v(\BC)$.  This intersection pairing has immediate applications to the Beilinson--Bloch height pairing on  $\Ch^*(X_K)^0$, extending Hriljac--Faltings' work for the N\'eron--Tate height pairing. Secondly, for Deligne's pairing,  Gillet and Soul\'e \cite{GS91} define a bigger arithmetic Chow group $\wh \Ch^*(X)$ without fixing 
 K\"ahler forms $\mu_v$ at archimdean place. Finally, in a series of papers \cite{BGS95, BGS97}, Bloch--Gillet--Soul\'e developed a non-archimedean Arakelov theory for  $X/\CO_K$ with strictly semistable reductions.  Assuming  Grothendieck's standard conjectures \cite[Conjectures 1, 2]{Gr}, they defined   harmonic forms using Laplacians
 \cite[Theorem 6]{BGS97}.  Using Bloch--Gillet--Soul\'e's harmonic forms,  K\"unnemann  defined an  Arakelov group \cite[\S3.6]{Ku98a},  and related it to the Beilinson--Bloch height pairing \cite[\S3.8]{Ku98a}
 
 This paper has achieved two primary goals for a Hodge theory in higher dimensional polarized arithmetic varieties. 
 The first is a new definition of harmonic forms for general regular polarized arithmetic varieties.
  We will use Lefschetz operators $\sL$ instead of Laplacian operators $\Delta$, which allows us to define harmonic forms in more general situations, even including cohomology cycles. This idea was inspired by the work of K\"unnemann in a series of papers    \cite{Ku94, Ku95, Ku98a, Ku98b}. The second is a new decomposition theorem for arithmetic Chow groups in the Lefschetz operator $\sL$. This decomposition theorem allows  to
to   define so-called $\sL$-liftings from $\Ch^*(X_K)$ to $\wh\Ch^*(X)$ extending the work of Beilinson and Bloch on 
homologically trivial cycles.
   As a consequence of our new constructions, we will prove the following two statements concerning relations between various standard conjectures 
   and height pairings:
  \begin{itemize}
  \item 
{\em The existence of Beilinson--Bloch pairing follows from Grothendieck's standard conjectures \cite[Conjectures 1, 2]{Gr}.}
\item
{\em Assume our local standard conjectures. Then two arithmetic standard conjectures on $\wh\Ch^*(X)$ by Gillet--Soul\'e and on $\Ch^*(X_K)^0$ 
by Beilinson are  equivalent to each other, as they are both equivalent to a standard conjecture on our Arakelov Chow group $\ol\Ch^*(X)$.}
\end{itemize}

As one application of the first statement, in a recent paper \cite{Zh21}, for the product $X=C\times S$ of a curve and a surface over a number field,  we construct unconditionally a  Beilinson--Bloch type height pairing  (\cite{Be, Bl}) for 
homologically trivial algebraic cycles on $X$. Then for an embedding $\phi: C\lra S$, we define an arithmetic diagonal cycle modified from the graph of $\phi$. This work extends the previous work of Gross and Schoen \cite{GS95} when $S$ is the product of two curves.

This work also arose from an attempt to understand conjectural  Gross--Zagier type formulas for the heights of special cycles of Shimura varieties 
in terms of special values of $L$-series, 
such as the Gan--Gross--Prasad conjecture \cite{GGP} and Kudla's program \cite{Kud}.  Consequently, our conditional construction provides some canonical arithmetic  
 special cycles on integral models of Shimura varieties, including Hecke correspondences. 
 For example, for a Shimura variety $X$ of  orthogonal (resp. unitary type) over  a totally real (resp.  CM)  field $F$, our construction gives canonical arithmetic lifting 
  generating a series of Kudla cycles. By  work of Yuan--Zhang--Zhang, Liu, and  Westerholt-Raum \cite{YZZ, Liu, BW}, 
  we have the following  unconditional result:

  \begin{itemize}
 \item  {\em The  $\sL$- lifting of generating series of Kudla's  cycles is   modular in  case of divisors or in   case  $F=\BQ$. }
  \end{itemize}

In the rest of this introduction, we give a more detailed outline of this work. 

\subsubsection*{Local cycles}
As a local setting, we will consider a flat and projective morphism $f: X\lra S$ where $S=\Spec R$ with $R$ a complete discrete valuation ring and $X$ is regular with an ample line bundle $L$. Then there are groups of algebraic cohomology cycles  $A^*(X_s), A_*(X_s)$ over the special fiber $i: X_s\lra X$.
These groups have an action by the  Lefschetz operator $\sL$ defined by the first Chern class of $L$. There is also a map connecting them:
$$ i^*i_*: A_{n+1-*} (X_s)\lra A^*(X_s).$$
We denote its image as  $A_\varphi^*(X_s)$, and its cokernel as 
$A_\psi^*(X_s)$. Thus there is an  exact sequence
$$0\lra A_\varphi^*(X_s)\lra A^*(X_s)\lra A_\psi ^*(X_s)\lra 0.$$
Our starting point is a set of extended  Grothendieck's standard conjectures \cite[Conjectures 1 and 2]{Gr}  for   $A_\varphi^*(X_s)$, and $A_\psi^*(X_s)$, Conjectures
\ref{conj-nd}, \ref{conj-lef}, \ref{conj-hod}.
Under Conjecture \ref{conj-lef} of Lefschetz type, one can prove in  Theorem \ref{thm-harac} that 
there are   unique splittings  of the above short exact sequences of $\sL$-modules:
$$A^*(X_s)=A_\varphi^*(X_s)\oplus \CA^*_\psi(X_s);$$
Inspired by  work of K\"unnemann for K\"ahler manifolds,  $\CA^*_\psi(X_s)$ is called the group of {\em harmonic forms}. Indeed, a  complex analogue of the  map $i^*i_*$ in arithmetic intersection theory of Gillet--Soul\'e is the following map:
$$\partial\bar\partial: \wt A^{p-1, p-1}(X_\BC):= A^{p-1, p-1}(X_\BC) /(\Im \partial+\Im\bar\partial)\lra A^{p, p}_{\closed}(X_\BC).$$
The decomposition according to Lefschetz operator coincides with Laplacian operator:
$$A^{p, p}_{\closed}(X_\BC)=\partial\bar\partial(\wt A^{p-1, p-1}(X_\BC))\oplus \CH^{p, p}(X_\BC)$$
where $\CH^{p, p}(X_\BC)$ is the space of harmonic forms of degree $(p, p)$.
See Theorem \ref{thm-arch} and Corollary \ref{cor-arch}.

Based on the work  of  Bloch--Gillet--Soul\'e, and K\"unnemann,
we will show that  Conjectures
\ref{conj-nd}, \ref{conj-lef}, \ref{conj-hod} hold when $X/S$ is strictly semistable so that all strata  satisfy   Grothendieck's  standard conjectures \cite[Conjectures 1, 2]{Gr}.
See Theorem \ref{thm-ssr}.

As the first  application, we use harmonic forms to define some canonical local height parings of algebraic cycles under  Conjecture \ref{conj-lef}. More precisely, let $\wh Z^*(X)$ be the space of cycles on $X$ modulo homologically trivial cycles supported on $X_s$. Then there are  maps
$$A_{n+1-*}(X_s)\overset {i_*}\lra \wh Z^*(X)\overset {\omega=i^*} \lra A^*(X_s)$$
where $\omega$ is called the {\em curvature map}. A cycle $z\in \wh Z^*(X)$ is called {\em admissible}, if the curvature is harmonic: $\omega (z)\in \CA^*_\psi(X_s)$. We let $\ol Z^*(X)$ denote the group of admissible cycles and call it the {\em Arakelov group of admissible  cycles}.
Then there is an exact sequence:
$$0\lra i_*A_{n+1-*}^\psi (X_s)\lra \ol Z^*(X)\lra Z^*(X_\eta)\lra 0,$$
where $A_{n+1-*}^\psi (X_s)$ is the kernel of $i^*i_*: A_{n+1-*}(X_s)\lra A^*(X_s)$.
If we further assume Conjecture \ref{conj-nd}, then the above sequence has an 
{\em Arakelov    lifting } $z\mapsto z^\Ara$ for $z\in Z^*(X_\eta)$ such that 
$z^\Ara-z^\zar=i_*g$ with $g\in A_{n+1-*}(X_s)$ perpendicular to $\CA_\psi^{n+1-*}(X_s)$.
Thus we get a well-defined {\em Arakelov  height pairing} for two disjoint cycles $z\in Z^i(X_\eta)$, $w\in Z^j(X_\eta)$ with $i+j=n+1$:
$$(z, w)_\Ara:=z^\Ara\cdot w^\Ara.$$
The analogue in K\"ahler manifold $(X_\BC, \omega)$ for an Arakelov lifting  $z^\Ara$ of a $z\in Z^*(X_\BC)$  is a pair $(z, g)$ with a current 
$g\in D^{p-1, p-1}/\Im \partial+\Im\partial$ such that 
$$\pp g=\delta _z-\omega_z,$$
where $\omega_z\in \CH^{p, p}(X_\BC)$ is a harmonic class representing $z$. 
See Gillet--Soul\'e \cite{GS91}.
The normalization means 
$$\int_{X_\BC}  gh=0, \qquad h\in \CH ^{n+1-p, n+1-p} (X_\BC).$$
The Arakelov  pairing of $z$ with a disjoint cycle $w$ with arithmetic complement degree is defined as
$$(z, w)_\Ara=\int _{X_\BC} g\delta _w.$$

As one byproduct, one has the following. 
Theorem  \ref{thm-bb}: {\em Assume  Grothendieck's  standard conjectures \cite[Conjectures 1, 2]{Gr}. Then 
a cycle $z$ on $Z^*(X_\eta)$ with trivial cohomology class in $H^{2*}(X_\bata)(*)$
will have a lifting $z^\BB$ with trivial cohomology class in $A^*(X_s)$. }

We will also study the following map of  cohomology  cycles,
$$\mu: H_{X_s}^*(X)\lra H^*(X_s), $$
and  denote its image and kernel  as $H_\varphi^*(X_s)$  and its cokernel as $H_\psi^*(X_s)$
to obtain an  exact sequence
$$0\lra H_\varphi^*(X_s)\lra H^*(X_s)\lra H_\psi^*(X_s)\lra 0.$$
Then we propose Conjecture \ref{conj-coh} of Lefschetz type for $H_\varphi^*(X_s)$ and  $H^*_\psi (X_s)$.
Under this conjecture, one has  the following statements:
\begin{enumerate}
\item Theorem \ref{thm-harcoh}: there is a    unique splitting  of above short exact sequence of $\sL$-modules:
$$H^*(X_s)=H_\varphi^*(X_s)\oplus \CH^*_\psi(X_s).$$
We call {\em $\CH^*_\psi(X_s)$ the space of  harmonic forms}.
\item Theorem \ref{thm-lef}: there is a   connection to the group of   invariant cycles given as follows:
$$H_\psi^*(X_s)\iso H^*(X_\bata)^{\Gal (\bata/\eta)}.$$
\end{enumerate}

For cohomology cycles,  we will show that our conjectures are equivalent to some conjectures about perverse sheaf cohomology, see Theorem \ref{thm-perv}. Based on the work of Beilinson--Bernstein--Deligne--Gabber,
we will show that our Conjecture \ref{conj-coh}  holds  when $R$ has  equal characteristics (Theorem \ref{thm-perv}, Corollary \ref{cor-bbdg}).

\subsubsection*{Global cycles}
As a global setting, we will consider a flat and projective morphism $f: X\lra S$ where $S=\Spec \CO_K$ with $K$ a number field and  $X$ is regular with an arithmetic ample line bundle $L$ (\cite{Zh}). Inside  Gillet--Soul\'e's arithmetic Chow group $\wh \Ch^*(X)$ with real coefficients, there is  an Arakelov Chow group 
$\ol\Ch^*(X)$ of admissible cycles with harmonic curvature everywhere.  This group fits in an exact sequence
$$0\lra B^*(X)\lra \ol \Ch^*(X)\lra \Ch^*(X)\lra 0$$
where $B^*(X)$ is the space of vertical cycles with trivial curvature.
This group has a 3-step filtration
$$F^i\ol\Ch^*(X)=\begin{cases} \ol\Ch^*(X), &\text{if $i\le 0$},\\
\ol\Ch^*(X)^0,&\text{if $i=1$},\\
B^*(X), &\text{if $i=2$}.
\end{cases}
$$ 
where $\ol\Ch^*(X)^0=\Ker (\ol \Ch^*(X)\lra H^{2*}(X_\bata)(*))$.
The associated graded quotients are 
$$G^i\ol\Ch^*(X)=\begin{cases}
A^*(X_\eta), &\text{if $i=0$,}\\
\Ch^*(X_\eta)^0, &\text{if $i=1$,}\\
B^*(X), &\text{if $i=2$,}
\end{cases}
$$
where $A^*(X_\eta)$ and $\Ch^*(X_\eta)^0$ are the image and the kernel respectively for the map $\Ch^*(X)\lra A^*(X_\eta)$. 
Assuming Gillet--Soul\'e's arithmetic standard conjecture and our local conjectures, 
 we will show  that there is a unique splitting of graded $\BR$-modules
$$\alpha: \bigoplus _{i=0}^2G^i\ol \Ch^*(X)\iso \ol \Ch^*(X)$$
such that $\alpha |G^1$ is $\sL$-linear, and $\alpha |G^0$ is $\sL$-linear  modulo $\alpha (G^2)$ and $\sLambda$-linear  modulo $\alpha (G^1)$,
see Theorem \ref{thm-ac-dec}.  We also show that  the Lefschetz operator $\alpha ^{-1} \sL\alpha $ 
is determined by an $\sL$-isomorphism $\beta: A^*(X_\eta)\lra B^{*+1}(X)$.  {\em The structure of the $\BR[\sL]$-module   $\ol \Ch^*(X)$ with a symmetric pairing depends only on the graded quotients and a mysterious  isomorphism $\beta$}. 
We give two applications of this splitting. 

The first one is  that Conjectures
\ref{conj-nd}, \ref{conj-lef}, \ref{conj-hod} imply some so called {\em $\sL$-lifting} for the projection $\ol\Ch^*(X)\lra \Ch^*(X_K)$, and thus a {\em $\sL$-
pairing} on $\Ch^*(X_K)$. We will define unconditionally the $\sL$- liftings for divisors $\Ch^1(X_K)$ and  $0$-cycles $\Ch^n(X_K)$ in 
Corollary \ref{cor-lef-pic}, \ref{cor-lef-0c}. As a consequences, we will have some canonical arithmetic liftings of the 
generating series of Kudla's divisors and $0$-cycles.

The second one is that  Conjectures
\ref{conj-nd}, \ref{conj-lef}, \ref{conj-hod} imply the equivalence (Theorem \ref{thm-gsbb})  between  the standard conjecture by Gillet--Soul\'e for arithmetic Chow groups $\wh \Ch^*(X)$ and the standard conjecture of Beilinson for homologically trivial 
Chow groups $\Ch^*(X)^0$.

In  function field case, we will also define the group $\ol H^*(X)$ of admissible cohomological cycles, and prove a decomposition Theorem \ref{thm-coh-dec} for cohomology group $H^*(X)$ :
 $$\alpha: \bigoplus _{i=0}^2G^i\ol H^*(X)\iso \ol H^*(X),$$
 such that $\alpha |G^1$ is $\sL$-linear, and $\alpha |G^0$ is $\sL$-linear  modulo $\alpha (G^2)$ and $\sLambda$-linear  modulo $\alpha (G^1)$. For an open embedding $j: U\to S$ such that $f$ is smooth, the above decomposition induces a
 decomposition of $\BQ_\ell$-vector spaces:
 $$\alpha: \bigoplus_{i=0}^2H^i(S, j_*R^{*-i}f_{U*}\BQ_\ell)\iso H^*(S, j_{!*}Rf_{U*}\BQ_\ell).$$
 Notice that our decomposition  is usually different than the one  induced from the canonical splitting  of the complex
 $Rf_{U*}\BQ_\ell$:
 $$Rf_{U*}\BQ_\ell=\bigoplus _{m\in \BZ} R^{m}f_{U*}\BQ_\ell [-m].$$

\subsection*{Acknowlegement}
We are grateful to  Klaus K\"unnemann and Weizhe Zheng for their help in preparing this paper,
especially to Klaus K\"unnemann for suggesting we relate our work to the previous work of Bloch--Gillet--Soul\'e and himself, and to Weizhe Zheng for helping us connect our conjecture to the weight monodromy conjecture and the decomposition theorems in perverse sheaves.  
We want to thank Yifeng Liu, Congling Qiu, Xinyi Yuan, and the referees for their valuable comments on an early draft of this paper. 
 This work is partially supported by a grant from the National Science Foundation.

\section{Local cycles}\label{sec-loc}

In this section, we first propose  some  conjectures (\ref{conj-nd}, \ref{conj-coh}, \ref{conj-lef},
\ref {conj-hod}) as extensions of  Grothendieck's standard conjectures  \cite[Conjectures 1, 2]{Gr} for smooth and projective varieties. 
Then we use these new conjectures to construct some canonical splittings of various groups of cycles (\ref{thm-harcoh}, \ref{thm-harac}).
By comparison with a result of K\"unnemann on complex varieties \cite{Ku95}, our splitting provides a non-archimedean analog of {\em harmonic forms} without using Laplacian operators. Our treatment applies to both cohomology cycles and algebraic cohomology cycles. 

For cohomology cycles,  we will show that our conjectures are equivalent to some conjectures about perverse sheaf cohomology, see Theorem \ref{thm-perv}. In particular, our conjectures of Lefschetz type hold when the base has equal characteristics by Beilinson--Bernstein--Deligne--Gabber \cite{BBDG} (\ref{cor-bbdg}).

For algebraic cycles on strictly semistable fibers, we will show that our conjectures are consequences of  Grothendieck's standard conjecture
by work of  Bloch--Gillet--Soul\'e \cite{BGS97} and K\"unnemann \cite{Ku98a} (\ref{thm-ssr}). Applying de Jong's alterations, we can prove that 
the  Beilinson--Bloch liftings for homologically trivial cycles exist under  Grothendieck's 
standard conjecture \cite[Conjectures 1, 2]{Gr} (\ref{thm-bb}).

\subsection{Cycles on a degenerate fiber}\label{ss-cc}
Let $S=\Spec R$  with $R$  a complete discrete valuation ring, with the generic point $\eta=\Spec K$ and the closed point $s=\Spec k$ with $k$ separately  closed. 
Let $f: X\lra S$ be a proper and flat morphism from a regular scheme with the special fiber $i: X_s\lra X$ and the generic fiber $j: X_\eta\lra X$.
Since $X$ is defined by finitely many equations, by approximation, we may assume that $X=X_0\otimes _{S_0}S$ , where $S_0=\Spec R_0$ with $R_0$ a complete discrete valuation subring of $R$ such that 
\begin{enumerate}
\item the residue field $k_0$ of $R_0$ is of finite type over its prime field;
\item $R$ is the completion of the maximal unramified extension of $R_0$.
\end{enumerate}

\subsubsection*{Algebraic cycles}

For $Y=X, X_\eta, X_s$, there are the Chow homology groups $\Ch_*(Y)$ with rational coefficients defined as the quotients of the free groups of integral subschemes modulo rational equivalence, and the Chow cohomology groups $\Ch^*(Y):=\Ch^*(Y\overset{id}\lra Y)$ defined as bivariant operations on the Chow homology groups on $Y$-schemes  as in  Fulton \cite[Definition 17.3]{Fu}.
Then $\Ch^*(Y)$ has a commutative ring structure which acts on $\Ch_*(Y)$ by  cap product:
$$\cap: \Ch^i(Y)\otimes \Ch_j(Y)\lra \Ch_{j-i}(Y).$$

If $Y=X_s, X_\eta$, then $Y$ is proper over a field.  Composing with the degree map $\deg: \Ch_0(Y)\lra \BQ$, there is  a pairing 
$$(\cdot, \cdot)_Y: \Ch^i(Y)\otimes \Ch_i(Y)\lra \BQ.$$

If $Y=X$ or $X_\eta$,  $Y$ is regular.  The cap  product with $[Y]$  defines  an isomorphism 
$$\cap [Y]: \Ch^i(Y)\iso \Ch_{\dim Y-i}(Y).$$ 

Let $n=\dim X_\eta$ be the relative dimension of $f$. Then there is  a composition of some morphisms of Chow groups:
$$
i^*i_*: \quad  \Ch_{n+1-i} (X_s)\overset {i_*} \lra \Ch_{n+1-i} (X)\simeq \Ch^i(X)\overset {i^*}\lra \Ch^i(X_s).
$$
One main goal of this paper is to study this map after taking their classes in an    $\ell$-adic cohomology for a prime  $\ell$ invertible over $S$.

\subsubsection*{Cohomology cycles}

We start with  the following localization sequence of the $\ell$-adic cohomology  and homology groups with $\BQ_\ell$-coefficients:
$$
\cdots \lra  H^{i-1}( X_{\eta})\lra   H^i_{X_s} (X)\overset {\mu }\lra H^i(X_s)\lra H^i(X_\eta)\lra \cdots.
$$
We define the groups  of {\em vanising}  and {\em nearby} cycles: 
$$H_\varphi ^i(X_s):=\Im (\mu ), \qquad H_\psi^i(X_s):=\Coker (\mu ). $$
Then there is  an exact sequence 
\begin{equation}\label{seq-coh}
0\lra H_\varphi ^i(X_s)\lra H^i(X_s)\lra H_\psi ^i (X_s)\lra 0.
\end{equation}

There is a perfect 
pairing between $H^*(X_s)$ and $H^*_{X_s}(X)$ by the composition of the following maps:
$$(\cdot, \cdot)_X: \quad H^i(X_s)\otimes H^{2n+2-i} _{X_s} (X)(n+1)\lra H^{2n+2}_{X_s} (X)(n+1)\overset\deg \lra \BQ_\ell.$$
Thus we may define the homology group of $X_s$ by
$$H_*(X)=H^{2n+2-i} _{X_s} (X)(n+1).$$
This pairing induces  perfect pairings on $H^*_\varphi$ and $H_\psi^*$ as follows:
$$(\cdot, \cdot)_\varphi:  H^i_\varphi ( X_s)\otimes H^{2n+2-i}_\varphi( X_s)(n+1)\lra \BQ_\ell, \qquad 
(\mu \alpha, \mu \beta)_\varphi=(\mu \alpha, \beta)_X, $$
$$ (\cdot, \cdot)_\psi: H^i_\psi (X_s)\otimes H_\psi ^{2n-i} (X_s)(n)\lra \BQ_\ell,\qquad 
(\alpha, \beta)_\psi=(\bar\alpha, \bar\beta \cap [X_s])_{X}$$
 where  $\bar\alpha\in H^i(X_s)$,  $\bar\beta\in H^{2n-i}(X_s)$ are liftings of $\alpha$ and $\beta$.

\subsubsection*{Weights}
Notice that  the cohomology groups $H_{X_s}^*(X)$, $H^*(X_s)=H^*(X)$,  $H^*(X_\eta)$, and $H^*(X_\bata)$ (where $\bata=\Spec \bar K$) have 
canonical weight filtrations with respect to $X_s$.  For the last three groups, we refer to \cite[\S2.2]{It05a}. For the first group, we use its duality to the second group. 
In the rest of this paper, we will not use the above precise constructions  of weights except two formulae (\ref{eq-hhh1}) and 
(\ref{eq-hhh2}). More precisely, assume that  $X_s $ is strictly  semistable  in the sense of de Jong \cite[\S1.26]{dJ} that 
in the  decomposition $X_s=\bigcup_{i=1}^r Y_i$ with $Y_i$ irreducible, for any non-empty subset $I\subset \{1, \cdots, r\}$, the strata $Y_I:=\bigcap_{i\in I}Y_i$ is smooth of dimension $n+1-|I|$. Thus $\BQ_\ell$ can be represented by the C\v ech complex $(C^*, d)$ with  $C^i:=\bigoplus_{|I|=i+1} \BQ_{\ell, Y_I}$ on $X_s$. And there is  a spectral sequence
$$E_1^{p, q}:=H^q(X_s, C^p))=\bigoplus_{|I|=p+1}H^q(Y_I) \imp H^{p+q}(X_s).$$
The canonical weight $q$ on $H^q(Y_I)$ induces a weight filtration on $X_s$. More precisely, by weight consideration, 
this spectral sequence degenerates at $E_2^{p, q}=H^p(H^q(X_s, C^*))$ which has pure weight $q$. Thus there is a unique weight filtration on $H^*(X)$ such that
$$\gr^W_pH^{p+q}(X)=H^q(H^{p}(X_s, C^*)).$$
For example, when  $q=0$, there are   the following formulae for the highest weight piece of $H^p(X)$:
\begin{equation}\label{eq-hhh1}
0\lra \gr^W_pH^p(X)\lra \bigoplus_i  H^p(Y_i)\lra \bigoplus_{i<j} H^p(Y_{ij}),
\end{equation}
Taking duality, we also have the lowest weight piece for $H_{X_s}^p(X)$:
\begin{equation}
\label{eq-hhh2}
\bigoplus_{i<j}  H^{p-4}(Y_{ij})(-2)\lra \bigoplus _i H^{p-2}(Y_i)(-1)\lra \gr^W_pH^p_{X_s}(X)\lra 0.
\end{equation}

\subsubsection*{Algebraic cohomology cycles}
Using filtration by weights $W_*$, there is  the class maps
\begin{equation}\label{seq-class}
\Ch_{n+1-*}(X_s)\lra \gr^W_{2*}H^{2*}_{X_s}(X)(*), \qquad \Ch^*(X_s)\lra \gr^W_{2*}H^{2*} (X_s)(*).
\end{equation}
For construction of these cycle maps, one can use  an  argument in Bloch--Gillet--Soul\'e \cite[Appendix]{BGS95} with  ``envelopes" replaced by de Jong's alternation
\cite{dJ}. The argument in \cite{BGS95} works with this replacement because of  $\BQ_\ell$-coefficients. In particular, when $X$ has strict semistable reduction, there is  the following analogue formula for Chow groups:
\begin{equation}\label{eq-ccc1}
0\lra \Ch^p(X_s)\lra \bigoplus_i  \Ch^p(Y_i)\lra \bigoplus_{i<j} \Ch^p(Y_{ij}),
\end{equation}
\begin{equation}
\label{eq-ccc2}
\bigoplus_{i<j}  \Ch_p(Y_{ij})\lra \bigoplus _i \Ch_p(Y_i)\lra \Ch_p(X_s)\lra 0.
\end{equation}
The class maps (\ref{seq-class}) induce maps from these exact sequences to the exact sequences (\ref{eq-hhh1})
and (\ref{eq-hhh2}).

 Let $A_*(X_s)$ and $A^*(X_s)$ be the images of  the maps (\ref{seq-class}). Then there is  a connection map 
$$i^*i_*: A_{n+1-*}(X_s)\lra A^*(X_s).$$
Again, we define the groups of vanishing and nearby cycles by
$$A_\varphi ^*(X_s):=\Im (i^*i_*), \qquad A_\psi^*(X_s)=\Coker (i^*i_*).$$
Then there is  an exact sequence
\begin{equation}
\label{seq-ch}0\lra A_\varphi^*(X_s)\lra A^*(X_s)\lra A_\psi ^*(X_s)\lra 0.
\end{equation}

The intersection pairing between $\Ch^i(X_s)$ and $\Ch_i(X_s)$ induces a pairing  
$$A^i(X_s)\otimes A_i(X_s)\lra \BQ$$
which is compatible with the pairing of cohomology groups. Moreover, the same process in cohomology groups defines 
the pairings on  $A_\varphi^*(X_s)$ and $A_\psi^*(X_s)$.

Thus there is  a morphism between above two sequences
$$\xymatrix{0\ar[r]&A_\varphi ^i(X_s)\ar[r]\ar@{^{(}->}[d]&A^i(X_s)\ar[r]\ar@{^{(}->}[d]&A_\psi ^i(X_s)\ar[r]\ar[d]&0\\
0\ar[r]&\gr^W_{2i}H_\varphi ^{2i}(X_s)(i)\ar[r]&\gr^W_{2i}H^{2i}(X_s)(i)\ar[r]&\gr^W_{2i}H_\psi ^{2i}(X_s)(i)\ar[r]&0.
}
$$
We will fix an ample line bundle $L$ over $X$. Let $\sL$ be the operator over each group in the above diagram defined by the cup product with the first Chern class $c_1(L)\in H^2(X_s)(1)$.

\subsubsection*{Standard conjectures}
We would like to propose the following analog of Grothendieck's standard conjectures \cite[Conjectures 1, 2]{Gr}:

\begin{conj}\label{conj-coh}
Let $n=\dim X_\eta$. 
\begin{enumerate}
\item For $i\le n$, there is  an isomorphism 
$$\sL^i: H^{n-i}_\psi (X_s)\iso H^{n+i} _\psi (X_s)(i).$$
\item For $i\le n+1$, there is  an isomorphism 
$$\sL^i : H_\varphi^{n+1-i} (X_s)\iso H_\varphi^{n+1+i}(X_s)(i).$$
\end{enumerate}
\end{conj}

\begin{conj}\label{conj-nd} The intersection pairing on algebraic cohomology classes $$A^*(X_s)\times A_*(X_s)\lra \BQ$$
is perfect.
\end{conj}

\begin{conj}
\label{conj-lef}
Let $n=\dim X_s$. 
\begin{enumerate}
\item For $i\le n/2$,  there is  an isomorphism 
$$\sL ^{n-2i}: A_\psi ^i(X_s)\iso A_\psi ^{n-i}(X_s) .$$
\item For $i\le {(n+1)/2}$,  there is  an isomorphism 
$$\sL ^{n+1-2i}: A_\varphi ^i(X_s)\iso A_\varphi ^{n+1-i}(X_s).$$
\end{enumerate}
\end{conj}

 \begin{conj}
\label{conj-hod}
Let $n=\dim X_s$.
\begin{enumerate}
\item For $i\le n/2$,   $0\ne x\in \Ker (\sL ^{n+1-i} |A_\psi ^i(X_s) )$, we have 
$$(-1)^i (x, \sL^{n-i}x)_\psi>0.$$
\item For $i\le {(n+1)/2}$,  $0\ne x\in \Ker (\sL^{n+2-i}|A_\varphi^i(X_s)$, we have 
$$(-1)^i (x, \sL ^{n+1-i}x)_\varphi> 0.$$
\end{enumerate}
\end{conj}

\begin{remark} We want to give some connections between the above conjectures and Grothendieck's standard conjectures \cite[Conjectures 1, 2]{Gr}.
\begin{enumerate}
\item  
If $X/S$ is smooth, then $A_\varphi^*(X_s)=H_\varphi ^*(X_s)=0$ and 
$$A_\psi^*(X_s)=A^*(X_s)=A_{n+1-*} (X_s), \qquad H_\psi^*(X_s)=H^*(X_s)=H_{X_s}^*(X_s).$$
Thus the above conjectures are the  Grothendieck conjectures for $X_s$.
\item Conversely if $X_s$ is strictly semistable, based on  work of Bloch--Gillet--Soul\'e \cite{BGS97} and K\"unnemann \cite{Ku98a}, we will show that 
 the Grothendieck standard conjectures \cite[Conjectures 1, 2]{Gr} for strata implies Conjectures \ref{conj-lef}, \ref{conj-hod}.
 See Theorem \ref{thm-ssr}. It will be interesting 	to extend these results to general situation.
 \end{enumerate}
\end{remark}

\begin{remark} Instead of working on $A^*(X_s)$ we can also work on $A_*(X_s)$ by 
defining 
$$A_*^\psi(X_s)=\ker (i^*i_*), \qquad A_*^\varphi (X_s)=\Im (i_*i^*)=A^{n+1-*}_\varphi (X_s)$$
to obtain an exact sequence
$$0\lra A_*^\psi (X_s)\lra A_*(X_s)\lra A_*^\varphi (X_s)\lra 0.$$
Then there are also  standard conjectures for $A_*(X_s)$ which is equivalent to those for $A^*(X_s)$.
Similar equivalence holds for $H^*(X_s)$ and $H_{X_s}^*(X_s)$.
\end{remark}

\subsubsection*{Harmonic forms}

Apply Proposition \ref{prop-lef-dec} to $\bigoplus _iH_?^*(X_s)(i)$, there is  the following splittings:

\begin{thm}\label{thm-harcoh}
Assume conjecture \ref{conj-coh}, then there is  a unique decomposition of $\BQ_\ell$-modules
$$H^*(X_s)=H_\varphi^*(X_s)\oplus \CH_\psi ^*(X_s),$$
so that the  induced morphism
$$\bigoplus _iH^*(X_s)(i)=\bigoplus_i H_\varphi^*(X_s)(i)\oplus \bigoplus _i\CH_\psi ^*(X_s)(i),$$
is  $\sL$-linear.
\end{thm}

The space $\CH_\psi^*(X_s)$ is called the space of {\em harmonic forms.}

\begin{thm}\label{thm-harac}
 Assume Conjecture   \ref{conj-lef}.
Then there is  a unique decomposition of $\sL$-modules
$$A^*(X_s)=A_\varphi^*(X_s)\oplus \CA_\psi^*(X_s).$$
\end{thm}

The space $\CA_\psi^*(X)$ is called the space of {\em harmonic forms.}

\begin{remark} These two decompositions can be considered as non-archimedean analogs of harmonic form decompositions, see Corollary \ref{cor-arch}.
\end{remark}

\subsection{Archimedean analogue}
We want to write an analogue sequence of (\ref{seq-ch}) when $K=\BC$ with valuation $|\cdot |$ 
using Gillet--Soul\'e's theory  \cite[ \S3.3.4]{GS91} of arithmetic Chow groups  over ``arithmetic ring $(\BC, |\cdot |)$".
Let $X$ be a proper and smooth complex variety. Then according to Bloch--Gillet--Soul\'e \cite{BGS95}, 
the analogue of ``$H^i_{X_s}(X)\lra H^i(X_s)$" in archimedean case is given by
$$\mu :=\pp: \wt A^{i-1, i-1} (X(\BC)) \lra A^{i, i}_{\closed}, \qquad  \wt A^{i-1, i-1} (X(\BC)):=\frac {A^{i-1.i-1}}{\Im \partial+\Im \bar  \partial }.$$
In this case, the analogue of sequence (\ref{seq-coh}) becomes
$$0\lra \Im(\partial \bar \partial)^{i, i}\lra A^{i, i}_{\closed}(X_\BC) \lra H^{i, i}(X(\BC) )\lra 0.$$
Thus we write this sequence as
$$0\lra H_\varphi^*(X_\BC)\lra H^*(X_\BC)\lra H_\psi^*(X_\BC)\lra 0.$$

The standard conjectures for $H^*_\psi (X_\BC)$ are the classical hard Lefschetz and Hodge index theorem in Hodge theory. For $H_\varphi^*(X_\BC):=\Im (\partial\bar\partial)^{i, i}$
we have the following:
\begin{thm}[K\"unnemann]
\label{thm-arch}
For $i\le (n+1)/2$, there is  an isomorphism 
$$\sL ^{n+1-2i}: H_\varphi^i(X_\BC) \iso H_\varphi ^{n+1-i}(X_\BC).$$
Moreover for $0\ne \alpha \in H_\varphi ^i(X_\BC)$, $\sL ^{n+2-2i} \alpha =0$,
then
$$(-1)^i(\alpha,\sL ^{n+1-2i} \alpha)>0.$$
\end{thm}
\begin{proof}
The first part is proved in \cite[Lemma 10.4]{Ku94}. The second part is proved in \cite[Theorem 1.2]{Ku95}.
\end{proof}

By Proposition \ref{prop-lef-dec}, there is  the following:

\begin{cor}\label{cor-arch}
There is a unique decomposition into $\sL$-modules:
$$A^{*, *}_{\closed} =\partial\bar\partial (A^{*-1, *-1})\oplus \CH^{*, *}(X).$$
\end{cor}
This decomposition is nothing but harmonic decomposition using the Laplacian operator.
Thus $\CH^*_\psi (X):=\CH^{*, *}(X)$ is the space of harmonic forms. To be consistent with the notation in the non-archimedean situation, we still denote this decomposition  as
$$H^*(X_\BC)=H_\varphi^*(X_\BC)\oplus \CH_\psi^*(X_\BC).$$

\subsection{Invariant cycles}

Now we want to connect the hard Lefschetz for $H_\psi^*$ to a conjecture about invariant cycles, which itself is a consequence of Deligne's weight monodromy conjecture \cite{De70}. 
 
\begin{thm}\label{thm-lef} Let $\bata=\Spec \bar K$ be the geometric point of $S$ with Galois group $I=\Gal (\bar K/K)$.
Then the following four statements are equivalent:
\begin{enumerate}
\item 
Conjecture \ref{conj-coh} of Lefschetz type  for $H_\psi^*(X_s)$;
\item The bijectivity  of the following composition of maps:
$$H_\psi^*(X_s)\lra H^*(X_\eta )\lra H^*(X_\bata)^I;$$
\item The surjectivity of the map to  invariant cycles: 
$H^*(X_s)\surj H^*(X_\bata)^{I}$;
\item For each $i$, $W_i H^i(X_\bata)^I=H^i(X_\bata)^I$.
\end{enumerate}
\end{thm}

\begin{proof}
 Consider the long exact sequence:
\begin{equation}\label{long-coh}
\cdots \lra H^{i-1}( X_{\eta})\lra  H_{X_s}^i(X)\overset {\mu }\lra H^i(X_s)\lra H^i(X_\eta)\lra \cdots.
\end{equation}
This sequence is self-dual with respect to the pairing:
$$H_{X_s}^i(X)\otimes H^{2n+2-i}(X_s)\lra \BQ_\ell (-1-n).$$
Notice that $H^i(X_s)$ has weight $\le i$ and $H_{X_s}^i(X)$ has weight $\ge i$, thus there is  an exact sequence: 
$$
0 \lra H^{i-1}(X_\eta )/W_{i-1}H^{i-1}(X_\eta)\lra  H^i_{X_s}(X)\overset {\mu  }\lra  H^i(X_s)\lra W_i H^i(X_\eta)\lra 0,
$$
which is self-dual with respect to $(i \longleftrightarrow 2n+2-i)$.
Thus there are isomorphisms:
$$W_iH^i(X_\eta)\iso H_\psi^i(X), \qquad H^i(X_\eta)/W_iH^i(X_\eta)\iso H_\psi^{2n+1-i}(X_s)^\vee (-n-1).$$
Combining these two isomorphisms, we get an exact sequence of $\BQ_\ell[\sL]$-modules:
\begin{equation}\label{short-psi}
0\lra H_\psi^*(X_s)\lra H^*(X_\eta)\lra  H_\psi^{2n+1-*} (X_s)^\vee (-n-1)\lra 0.
\end{equation}

As the operator $\sL$ is an even degree operator,  we can decompose this sequence into two sequences 
\begin{equation}\label{short-psi-e}
0\lra H_\psi^{2*}(X_s)\lra H^{2*}(X_\eta)\lra H_\psi^{2n+1-2*}(X_s)^\vee (-n-1)\lra 0,
\end{equation}
\begin{equation}
\label{short-psi-o}
0\lra  H_\psi^{2*+1}(X_s)\lra H^{2*+1}(X_\eta)\lra H_\psi^{2n-2*}(X_s)^\vee (-n-1)\lra 0.
\end{equation}

Now we want to consider  the spectral sequence 
$$E_2^{p, q}:=H^p(I, H^q(X_\bata))\Rightarrow H^{p+q}(X_\eta).$$
Since the action of $I$ on $H^*(X_\bata)$ restricting to an open subgroup $I_0$ factors through the tame quotient $\BZ_\ell (1)$, this sequence degenerates. Thus there is  
 an exact sequence
$$0\lra H^1(I, H^{*-1}(X_\bata))\lra H^*(X_\eta)\lra H^0(I, H^*(X_\bata))\lra 0.$$
Using following  identities 
$$H^1(I, H^{*-1}(X_\bata))=H^{*-1}(X_\bata)_I(-1), \qquad H^0(I, H^*(X_\bata))=H^*(X_\bata)^I, $$
 we obtain  an exact sequence
\begin{equation}\label{short-gal}
0\lra H^{*-1}(X_\bata)_I(-1)\lra H^*(X_\eta)\lra H^*(X_\bata)^I\lra 0.
\end{equation}
We can decompose it into even and odd degrees as well:
$$0\lra H^{2*-1}(X_\bata)_I(-1)\lra H^{2*}(X_\eta)\lra H^{2*}(X_\bata)^I\lra 0,$$
$$0\lra H^{2*}(X_\bata)_I(-1)\lra H^{2*+1}(X_\eta)\lra H^{2*+1}(X_\bata)^I\lra 0.$$

By Deligne \cite{De80}, $H^*(X_\bata)$ satisfies the hard Lefschetz. 
More precisely,  four end terms are Lefschetz modules with different centers:
$$H^{2*-1}(X_\bata)_I(-1): \quad (n+1)/2, \qquad H^{2*}(X_\bata)^I: \quad n/2,$$
$$H^{2*}(X_\bata)_I(-1): \quad n/2, \qquad H^{2*+1}(X_\bata)^I: (n-1)/2.$$

Now  applying Proposition \ref{prop-lef-dec} part 2 to the above two sequences, we get  unique splittings, 
\begin{equation}
\label{short-gal-e}
H^{2*}(X_\eta)\simeq H^{2*-1}(X_\bata)_I(-1)\oplus H^{2*}(X_\bata)^I.
\end{equation}
\begin{equation}\label{short-gal-o}
H^{2*+1}(X_\eta)\simeq H^{2*}(X_\bata)_I(-1)\oplus H^{2*+1}(X_\bata)^I.
\end{equation}
From  (\ref{short-gal-e}) and   (\ref{short-psi-e}), we get morphisms of graded modules 
$$\alpha\in \Hom (H^{2*}(X_\bata)^I, H^{2n+1-2*}_\psi (X_s)^\vee (-1-n)), \qquad \beta \in \Hom (H^{2*}_\psi (X_s), H^{2*-1}(X_\bata)_I(-1)).$$

Now we want to prove that part 1 implies part 2.
Assume $H_\psi^*(X)$ satisfies the hard Lefschetz \ref{conj-coh}, then four end terms are all Lefschetz modules with various different centers:
$$H_\psi^{2*}(X_s):  \quad n/2, \qquad H_\psi^{2n+1-2*}(X_s)^\vee: \quad (n+1)/2 ,$$
$$H_\psi^{2*+1}(X_s):\quad (n-1)/2, \qquad  H_\psi^{2n-2*}(X_s)^\vee:  \quad n/2.$$
Considering their centers, by Proposition \ref{prop-lef-dec} part 1, these two group homomorphisms $\alpha $, $\beta$ vanish. Thus the sequences (\ref{short-psi-e}) and (\ref{short-psi-o}) split with isomorphisms:
$$H_\psi^{2*}(X_s)\iso H^{2*}(X_{\bar\eta})^I, \qquad H^{2*-1}(X_{\bar\eta})_I(-1)\iso H_\psi^{2n+1-2*}(X_s)^\vee (-n-1),$$
$$H_\psi^{2*+1}(X_s)\iso H^{2*+1}(X_{\bar\eta})^I,\qquad H^{2*}(X_{\bar\eta})_I(-1)\iso H_\psi^{2n-2*}(X_s)^\vee (-n-1).$$
Combing these splittings, we get the splitting for sequences (\ref{short-psi}) and the isomorphisms:
$$H_\psi^{*}(X_s)\iso H^{*}(X_{\bar\eta})^I, \qquad H^{*-1}(X_{\bar\eta})_I(-1)\iso H_\psi^{2n+1-*}(X_s)^\vee (-n-1).$$
 In particular, we have part 2 of the theorem. 

It is clear that part 2 implies part 3 and that part 3 implies part 4. 

Now assume part 4.  By duality, the $H^i(X_\bata)_I$ has weight $\ge i$. Thus the sequence (\ref{short-gal}) splits according to weight comparison. 
In particular, we have 
$$H_\psi^i(X_s)=W_iH^i(X_s)\simeq H^i(X_\bata)^I.$$
Now the hard Lefschetz for  $H^*(X_\bata)$ proved by Deligne \cite{De80} gives the hard Lefschetz for $H^*_\psi (X_s)$.
Thus we have completed  the proof of theorem. 
\end{proof}

\begin{remark} 
Combined with known cases of the weight monodromy conjecture, Conjecture \ref{conj-coh} holds for $H_\psi^*(X_s)$ in the following cases:
\begin{enumerate}
\item $X/S$ is smooth, see Deligne \cite{De80};
\item $X_\eta$ is a curve, or  an abelian variety [SGA7];
\item $X_\eta$ is a surface,  see Rapoport--Zink  \cite{RZ} for semistable case, and de Jong's alteration \cite{dJ} for the general case;
\item $K$ has positive characteristic, see Deligne   \cite{De80} for $k_0$ a finite field, and Ito  \cite{It05a} for the general case;
\item $k$ has characteristic $0$, see Steebrink \cite{St} and Saito  \cite{Sa};
\item $X_\eta$ is a set-theoretically complete intersection in a toric variety. see Scholze \cite{Sc};
\item $X$ has a uniformization by Drinfeld  upper half spaces \cite{It05b}.
\end{enumerate}
\end{remark}

For $A_\psi^*(X)$ there is  a slightly weaker result:
\begin{thm} \label{thm-lef-ac} Assume either the smoothness of $X/S$  or both Conjectures \ref{conj-nd} and \ref{conj-lef} for $A_\psi^*(X_s)$.
Then the map 
$$A_\psi ^*(X_s)\lra \gr^W_{2*} H^{2*}(X_\bata)^I(*)$$
is injective. 
\end{thm}
\begin{proof}
If $X/S$ is smooth, then 
$$A_\psi^*(X)=A^*(X_s)\subset H^{2*}(X_s)(*)=H^{2*}(X_\bata)^I.$$

If Conjectures  \ref{conj-nd} and \ref{conj-lef} for $A_\psi^*(X_s)$ hold, then the class map 
$$A_\psi ^i(X_s)\lra \gr^W_{2i}H_\psi ^{2i}(X_s)(i)$$
is injective. This follows from the fact that the map 
respects to the pairing. 

We consider the embedding from  equation (\ref{short-gal-e}), 
 $$
A_\psi ^*(X_s)\emb \gr ^W_{2*} H^{2*}(X_\eta)(*)\simeq \gr^W_{2*}H^{2*-1}(X_\bata)_I(*-1)\oplus \gr^W_{2*}H^{2*}(X_\bata)^I(*).
$$
The composition with the first projection gives a map 
$$A_\psi ^*(X_s)\lra \gr^W_{2*}H^{2*-1}(X_\bata)_I(*-1).$$
This is a morphism between two Lefschetz modules with centers $(n+1)/2$ and $n/2$.
Thus by Proposition \ref{prop-lef-dec}, this map must vanish. 
Thus we have the injectivity  in the theorem. 
\end{proof}

\subsection{Perverse decompositions}

In the following, we want to give an interpretation of our conjectures in terms of  perverse cohomology of the complex $Rf_*\BQ_\ell\in D_c^b(S)$ 
 as defined in \cite{BBDG}. Recall that by definition a perverse sheaf $F$ on $S$  is  a complex in $D_c^b(S)$ such that both $F$ and $\RD(F)$ are in $D^{[-1, 0]}_c(S)$, where $\RD$ is the Verdier duality operator on 
 $Rf_*\BQ_\ell\in D_c^b(S)$ defined by
 $$\RD(F)=R\CH om(F, \BQ_\ell[2] ).$$
 It can be shown that  for any complex $F\in D^{[-1, 0]}(S)$, $F$ is perverse if and only if the following conditions hold: \begin{enumerate}
 \item $\CH^{-1}F$ is a ``torsion free" sheaf on $S$ in the  sense that  the morphism $\CH^{-1}F\lra j_*j^*\CH^{-1}F$ is injective, 
 \item  $\CH^0F$ is a ``torsion sheaf" in the sense that $j^*\CH^0F=0$ or equivalently $\CH^0F=i_*i^*\CH^0F$.
 \end{enumerate}
 Any perverse sheaf is a successive extension of simples sheaves which have forms $i_*U_s$ or $j_*V_\eta[-1]$
 for simple sheaves $U_s$, $V_\eta$ at $s$ and $\eta$.

 The first result is the following decomposition theorem which we have learned from Weizhe Zheng \cite{ZWZ}:
 
 \begin{lem}\label{lem-dec-pH} For any $C\in D_c^b(S)$, there is  a  decomposition of complexes:
 $$C\simeq \bigoplus _{m\in \BZ} \pH^m C[-m].$$
 \end{lem}
 \begin{proof}
 It suffices to to show that for any integer $n$, $\Ext^1({^p\!\tau^{>n}}C, {^p\!\tau^{\le n}}C)=0$.
 Since $C$ is a successive extension of $\pH^mC[-m]$, it suffices to show that $\Ext^m(F, G)=0$
 for all perverse sheaves $F, G$ and $m\ge 2$. We may even reduce the following  three situations:
 $$F=i_*\BQ_\ell, \qquad F=\BQ_\ell[1], \qquad F=j_!V_\eta[1],$$
 where $V_\eta$ is non-constant simple sheaf at $\eta$.
  Since $Ri^! G=Di^*DG$, we have
 $$RHom (i_*\BQ_\ell, G)=RHom (\BQ_\ell, Ri^!G)\in D^{[0, 1]},$$
 $$RHom (\BQ_\ell[1], G)=RHom (\BQ_\ell[1], i^*G)\in D^{[0, 1]},$$
 $$RHom (j_!V_\eta[1], G)=RHom(V_\eta[1], j^*G)\in D^{[0, 1]}.$$
 It follows that $RHom (F, G)\in D^{[0, 1]}$. Thus 
 $$\Ext^m(F, G)=H^m (RHom (F, G))=0, \qquad\forall m\ge 2.$$
 \end{proof}
 
 \begin{remark}
 Unlike the usual cohomolgy splitting of complexes,  the perverse cohomology splitting in Lemma \ref{lem-dec-pH} may  not be unique 
 as there could be nontrivial elements in 
 $$\Hom (\pH^m C[-m], \pH^{m+1} C[-m-1])\iso \Hom (\CH^0(\pH^m C), \CH^{-1}(\pH^{m+1} C)).$$
 \end{remark}
 
 Applying  the sheaf cohomology to the identity in the lemma, we obtain a decomposition:
$$\CH^i(C)=\CH^{-1}(\pH^{i+1}C)\oplus \CH^0(\pH^iC).$$
 From this identity, it is clear that $\CH^0(\pH^iC)$ is the maximal torsion subsheaf of $\CH^i(C)$. Thus we have 
 
 \begin{equation}\label{seq-pH0}
 \CH^0(\pH^iC)=\Ker (\CH^i(C)\lra j_*j^*\CH^i(C)), 
 \end{equation}
 \begin{equation}\label{seq-pH1}
\CH^{-1}(\pH^{i+1}C)=\Im (\CH^i(C)\lra j_*j^*\CH^i(C)).
\end{equation}

In the following, we apply the above decomposition to the complex $Rf_*\BQ_\ell$.
We make the following conjectures:

\begin{conj}\label{conj-perv-dec}
On $S$,  there is  a  splitting of complexes:
$$\pR^if_*\BQ_\ell \iso \CH^0(\pR^if_*\BQ_\ell )\oplus \CH^{-1}(\pR^if_*\BQ_\ell)[1].$$
Moreover $\CH^{-1}(\pR^if_*\BQ_\ell)\iso j_*j^*R^{i-1}f_*\BQ_\ell$.
\end{conj}

\begin{conj}\label{conj-perv-lef}
For any $i\le n+1$, there is  an isomorphism
$$\sL^i : \pR^{n+1-i} f_*\BQ_\ell \iso  \pR^{n+1+i}f_*\BQ_\ell (i).$$
\end{conj}

\begin{remark} These two conjectures could be extended to a more general situation. More precisely, the statements in \cite[Theorem 6.2.5 6.2.10]{BBDG} should hold for any proper morphism of schemes.
\end{remark}

\begin{thm}\label{thm-perv} The  conjecture of Lefchetz type \ref{conj-coh} for $H_\psi^*(X)$ is equivalent to  Conjecture \ref{conj-perv-dec}. Assume Conjecture  \ref{conj-perv-dec}, then  Conjecture \ref{conj-coh} of Lefchetz type  for $H_\varphi^*(X)$ is equivalent to  Conjecture \ref{conj-perv-lef}. 
 \end{thm}
 
 \begin{proof} By Theorem \ref{thm-lef}, Conjecture \ref{conj-coh} of Lefchetz type  for $H_\psi^*(X)$ is
 equivalent to the surjectivity of the morphisms of sheaves  for each $i$ :
 $$R^if_*\BQ_\ell\lra j_*j^*R^ij_*\BQ_\ell.$$
 By the exact sequence (\ref{seq-pH1}), this is equivalent to the bijectivity
 $$\CH^{-1}(\pR^if_*\BQ_\ell)\iso j_*j^*\CH^{-1}(\pR^if_*\BQ_\ell).$$

 Notice that  
 $$\RD(\pR^if_*\BQ_\ell )=\pR^{2n+2-i}f_*\BQ_\ell (1+n) .$$
  Thus for the first part of theorem, it suffices to prove the following lemma which we have learned again from Weizhe Zheng 
  \cite{ZWZ}:
 \begin{lem} Let $F$ be a perverse sheaf on $S$. Then $F$ is split in the following sense
 $$F=\CH^0(F)\oplus \CH^{-1}(F)[1]$$
 if and only if 
 $$\CH^{-1}\RD (F) \iso j_*j^*\CH^{-1}\RD (F).$$
 \end{lem}
 \begin{proof} 
 Write exact sequences for $F$ and $\RD (F)$ using their cohomology 
 \begin{equation}\label{eq-hfh}
 0\lra \CH^{-1}(F)[1]\lra F\lra \CH^0(F),\lra 0\end{equation}
 \begin{equation}\label{eq-hdfh}
 0\lra \CH^{-1}\RD ( F)[1]\lra \RD (F)\lra \CH^0\RD( F)\lra 0,\end{equation}
 Write $U=i^*\CH^0\RD F$ and  $V=j^*\CH^{-1}\RD (F)$  as sheaves  at $s$ and $\eta$ respectively.

If $\CH^{-1}\RD (F) \iso j_*j^*\CH^{-1}\RD (F)=j_*V$, then apply $\RD$ to sequence \ref{eq-hdfh} to  get 
$$0\lra i_*U^D\lra F\lra j_*V^D[1]\lra 0$$
where
 $$U^D:=\Hom (U, \BQ_\ell ), \qquad V^D:=\Hom (V, \BQ_\ell(1)).$$
 Taking sheaf cohomology $\CH^*$, we obtain a surjective map
 $$i_*U^D\surj \CH^0(F).$$
 Taking global sections, this gives $U^D\surj i^*\CH^0(F)$. Any section of this map will provide a splitting of (\ref{eq-hfh}).

  Conversely, assume $F$ is split, then there is  a composition of the following surjective morphisms of perverse sheaves:
 $$F\lra \CH^{-1}F[1]\lra j_*j^*\CH^{-1}F[1].$$
 This is a surjective morphism in the category of perverse sheaves with the kernel of the form $i_*X$ with $X$ a vector space concentrated at degree $0$.
 By duality, we have 
 $$0\lra j_*j^*\RD \CH ^{-1}F [1]\lra \RD F\lra i_*X^D\lra 0.$$ 
 This implies that 
 $$\CH^{-1}\RD F=j_*j^*\RD \CH ^{-1}F.$$
 \end{proof}

It remains to prove the second part of theorem. Notice that the  hard Lefschetz for $\pR^*f_*\BQ_\ell$ is equivalent to the hard Lefschetz at $(\pR^*f_*\BQ_\ell)_\eta$ and $(\pR^*f_*\BQ_\ell)_s$. 
At $\eta$, it is the hard Lefschetz on $H^*(X_\bata)$ proved by  by Deligne \cite{De80}.
At $s$, under   Conjecture \ref{conj-coh} of Lefchetz type  for $H_\psi^*(X)$, by Theorem \ref{thm-lef},  formulae
(\ref{seq-pH0}, \ref{seq-pH1}), we have that 
$$\CH^{-1}(\pR^{i+1}f_{s*}(\BQ_\ell))=H_\psi^i(X), \qquad \CH^0(\pR^{i}f_{s*}(\BQ_\ell))=H_\varphi^i(X).$$
Thus there is  an exact sequence
$$0\lra H_\psi^{*-1}(X)[1]\lra (\pR^*f_*\BQ_\ell)_s\lra H_\varphi^*(X)\lra 0.$$
By the assumption in the theorem, the hard Lefschetz holds for $H_\psi^{*-1}(X)[1]$ with  center $n+1$. Thus
the hard Lefschetz for $(\pR^*f_*\BQ_\ell)_s$ with  center  $n+1$ is equivalent to the hard Lefschetz for $H_\varphi^*(X)$ with center $n+1$.
\end{proof}

By Beilinson--Bernstein--Deligne--Gabber \cite{BBDG}, we have the following:
\begin{cor}\label{cor-bbdg}    Assume that $S$ has equal characteristic, then  conjectures \ref{conj-coh},  \ref{conj-perv-dec}, and \ref{conj-perv-lef} 
all hold.
\end{cor}
\begin{proof} When $k$ is of characteristic $0$, these are special cases of \cite[Theorem 6.2.5, Theorem 6.2.10]{BBDG}. When $k$ has characteristic $p$, then the same proof over $\BQ$
will reduce the problems to the statements for varieties defined over finite fields: \cite[Corollary 5.4.7,  and  Theorem 5.4.10]{BBDG}.
\end{proof}

Notice that the results over finite fields were first proved by Deligne \cite{De80}.

\subsection{Strict semistable reductions}
As one attempt to deduce  our extended standard conjectures from   Grothendieck's  standard conjectures
\cite[Conjectures 1, 2]{Gr},
we have the following general result based on previous work of 
Bloch--Gillet--Soul\'e \cite{BGS97} and K\"unnemann \cite{Ku98a}:
 \begin{thm}\label{thm-ssr} Let $f: X\lra S$ be as in \S\ref{ss-cc}.
 Assume that $X$ has strictly semistable reductions and that on each stratum $Y_I$ of dimension $n_I=n+1-|I|$, the group $A(Y_I)$ of algebraic cohomology cycles satisfies Grothendieck's standard conjectures \cite[Conjectures 1, 2]{Gr}.
  Then  Conjectures \ref{conj-nd},  \ref{conj-lef}, and  \ref{conj-hod} hold. 
  \end{thm}
 
 \begin{proof} We will use results in  \cite{BGS97} where there are  different  definitions of $A^*(Y)$ and $A_*(Y)$. To avoid confusion, we denote their groups as $A^*(X_s)_{BGS}$ and $A_*(X_s)_{BGS}$. Recall that these groups are defined 
 using C\v ech complexes as  (\ref{eq-hhh1}), (\ref{eq-hhh2}), (\ref{eq-ccc1}) and (\ref{eq-ccc2}):
 \begin{equation}\label{eq-aaa1}
 0\lra A^*(X_s)_{BGS}\lra \bigoplus _i A^*(Y_i)\lra \bigoplus _{i\le j}A^*(Y_{ij}),\end{equation}
 \begin{equation}\label{eq-aaa2}
 \bigoplus _{i\le j}A_* (Y_{ij})\lra A_*(Y_i)\lra A_*(X_s)_{BGS}\lra 0,\end{equation}
 where for smooth variety $Z$ over a field of dimension $d$, $A_*(Z)$ is defined to be $A^{d-*}(Z)$.
 There are morphisms among various exact sequences:
 $$(\ref{eq-ccc1})\surj (\ref{eq-aaa1})\lra (\ref{eq-hhh1}), \qquad (\ref{eq-ccc2})\surj (\ref{eq-aaa2})\lra (\ref{eq-hhh2}).$$
 The composition are the class maps (\ref{seq-class})  of Chow groups. Thus there are  maps
 \begin{equation}\label{eq-aaaa}
 A_*(X_s)_{BGS}\surj A_*(X_s), \qquad A^*(X_s)_{BGS}\iso A^*(X_s).
 \end{equation}
 In the exact sequences (\ref{eq-aaa1}) and (\ref{eq-aaa2}), by assumption on the Grothendieck standard conjectures
 \cite[Conjectures 1, 2]{Gr}, $A^*(Y_i)$ is dual to $A_*(Y_i)=A^{n-i}(Y_i)$ and $A^*(Y_{ij})$ is dual to $A_*(Y_{ij})=A^{n-1-*}(Y_{ij})$. It follows that the intersection pairing on $A^*(Y_i)$ induces a perfect pairing 
 $$A^*(X_s)_{BGS}\times A_*(X_s)_{BGS}\lra \BR.$$
 Since this pairing is compatible with pairing on $A^*(X_s)\times A_*(X_s)$,
 we have that $A_*(X_s)_{BGS}\iso A_*(X_s)$. Thus we have proved Conjectures \ref{conj-nd}.

 For proving the rest of the theorem, we notice that the maps (\ref{eq-aaaa}) are compatible with connection map $i^*i_*$, i.e., the following diagram is commutative:
 $$\xymatrix
 {
 A_{n+1-*}(X_s)_{BGS} \ar[r]^{i^*i_*}\ar[d]^\wr&A^*(X_s)_{BGS}\ar[d]^\wr\\
 A_{n+1-*}(X_s)\ar[r]^{i^*i_*}&A^*(X_s).
 }
 $$

 We define the groups $A_\varphi^*(X_s)_{BGS}$ and $A_\psi^*(X_s)_{BGS}$ as image and cokernel of $i^*i_*$ on the top row analogously. Then there is an isomorphism  of two exact sequences:
 $$\xymatrix
 {
 0\ar[r]&A_\varphi^*(X_s)_{BGS}\ar[r]\ar[d]^\wr &A^*(X_s)_{BGS} \ar[r]\ar[d]^\wr&A^*_\psi(X_s)_{BGS}\ar[r]\ar[d]^\wr &0\\
 0\ar[r]&A_\varphi^*(X_s)\ar[r]&A^*(X_s) \ar[r]&A^*_\psi(X_s)\ar[r]&0.
 }
 $$
 
 The main results of Bloch--Gillet--Soul\'e \cite[Theorem 6]{BGS97} and K\"unnemann \cite[Theorem 2.17]{Ku98a} are that $A_\psi ^*(X_s)_{BGS}$ and $A_\varphi^*(X_s)_{BGS}$  both satisfy  Conjectures \ref{conj-lef}  and \ref{conj-hod}. \end{proof}

\begin{remark}\label{rem-sss} Here are some examples where the assumption of the theorem hold:
\begin{enumerate}
\item $X_\eta$ is a curve or a surface;
\item $X_\eta$ is an abelian variety with totally degenerate fiber: $X_s$ is a union of toric varieties, \cite[ \S 3.4]{Ku98a};
\item $X$ is the quotient of a Drinfeld upper half space with $\sL$ induced from  the canonical bundle, \cite[Proposition 4.4.]{It05b}.
\end{enumerate}
\end{remark}

\begin{remark}
In terms of their group $A^*(X_s)_{BGS}$, the {\em harmonic decomposition}  was already given by Bloch--Gillet--Soul\'e \cite[Theorem 6, part (iii)]{BGS97} using a Laplacian operator. 
\end{remark}

\subsection{Admissible cycles}
Let $f: X\lra S$ be as in \S \ref{ss-cc}. Then there is a chain of maps between various  cycles:
$$Z_{n+1-*} (X_s)\overset {i_*} \lra Z_{n+1-*} (X)=Z^*(X)\lra \Ch^*(X)\lra \Ch^*(X_s)\overset {i^*} \lra A^*(X_s).$$
We want to modify all  cycle groups   by modulo the images of  the kernel of the following map
$$Z_{n+1-*}(X_s)\lra A_{n+1-*} (X_s)$$ to obtain a new chain of maps:
\begin{equation}\label{eq-ici}
A_{n+1-*} (X_s)\overset {i_*} \lra \wh Z_{n+1-*} (X)=\wh Z^*(X)\lra \wh \Ch^*(X)\overset{i^*} \lra A^*(X_s).\end{equation}
The $i^*$  is called  the {\em curvature map} and denoted it by $\omega $.

Let $L$ be an ample line bundle, and  assume that Conjecture \ref{conj-lef} holds. Then by Theorem \ref{thm-harac}, and Corollary  \ref{cor-arch},
 there is a decomposition 
\begin{equation}
\label{eq-har-dec}
A^*(X_{s})=A_\varphi^*(X_{s})\oplus \CA^*_\psi(X_{s}),
\end{equation}
where $\CA_\psi^*(X_s)$ is the space of harmonic forms. 

We say that a class  $x\in \wh Z^*(X)$ (resp, $\wh \Ch^*(X)$) is {\em admissible}, if its curvature $\omega(x)$ is harmonic. 
Let $\ol Z^*(X)$ (resp. $\ol \Ch^*(X)$) denote the group of admissible classes called the {\em Arakelov group}.
It is clear that $A_\varphi^*(X_s)$ is the image $i^*i_*$ in (\ref{eq-ici}).  Thus every class in $\wh Z^*(X)$ can be modified to be admissible by adding a vertical cycle on a special fiber. 
Denote  $A_*^\psi (X_s)$ as the kernel of $i^*i_* : A_*(X_s)\lra A^{n+1-*} (X_s)$.
Then there is   an exact sequence:
\begin{equation}\label{seq-BZZ}
0\lra i_*A_{n+1-*}^\psi (X_s)\lra \ol Z^*(X)\lra Z^*(X_\eta)\lra 0.
\end{equation}

\subsubsection*{Arakelov liftings}
We want to define some canonical lifting for the sequence \ref{seq-BZZ}. 
For any  cycle $z\in Z^*(X_\eta)$ with  Zariski closure $z^\zar$, an admissible  lifting $z^\Ara$ of $z$  is called {\em Arakelov  lifting} if the difference $z^\Ara-z^\zar=i_*g$ for some $g\in A_{n+1-*}(X_s)$ that  is perpendicular to the image of $\CA_\psi^{n+1-*}(X_s)\subset A^{n+1-*} (X_s)$ in the space of harmonic forms. 

\begin{thm}\label{thm-adm} Assume either the smoothness of $X/S$ or Conjectures \ref{conj-nd} and \ref{conj-lef}. 
Then for any  cycle $z\in Z^*(X_\eta)$, the Arakelov  lifting $z^\Ara$ of $z$ exists and is unique. 
\end{thm}
\begin{proof} 
If $X/S$ is smooth, then we simply take $z^\Ara=z^\zar$. So we can assume  Conjectures \ref{conj-nd}, \ref{conj-lef} in the following. 
Let's start with  the following  exact sequence
$$0\lra A^\psi _{n+1-*} (X_s)\lra A_{n+1-*}(X_s)\overset{i^*i_*} \lra A^*(X_s)\lra A_\psi ^*(X_s)\lra 0.$$
By Conjecture 2.1.2,   the above sequence   is dual to the same sequence with $*$ replaced by $n+1-*$. The decomposition \ref{eq-har-dec} implies a dual decomposition 
$$A_*(X_s)=A_*^\psi (X_s)\oplus \CA^\varphi _* (X_s).$$
The $\CA^\varphi _*(X_s)$ is in fact the orthogonal complement of $\CA_\psi ^* (X_s)$. 

Now for any $z\in Z^*(X_\eta)$, and any lifting $\bar z\in \ol Z^*(X_\eta)$, their difference has an expression $\bar z-z^\zar=i_*g$ for some $g\in A_{n+1-*}(X_s)$. We can modify this element by adding some element in $A^\psi _{n+1-*}(X_s)$ so that it  belongs to $\CA^\varphi _{n+1-*}(X_s)$. 
\end{proof}

As an application of Proposition  \ref{thm-adm},   for two disjoint cycles $z_1\in Z^i(X_\eta)$ and $z_2\in Z^j(X_\eta)$ with $i+j=n+1$, we can 
define their {\em Arakelov  height  pairing} as follows:
$$(z_1, z_2)_\Ara = z_1^\Ara \cdot z_2^\Ara =z_1^\Ara \cdot z^\zar _2.$$

The archimedean analog of the above construction is classical due to Arakelov, Faltings, and Gillet--Soul\'e. For a smooth, complex projective variety $X$ with a K\"ahler form $\omega$,  we can extend any cycle $z\in Z^*(X)$ to a Green current $g$ so that $\pp g=\delta _z-h_z$ where $h_z$ is the harmonic form representing $z$.
This current is unique up to an addition of a harmonic form. We may further normalize this current by requiring that this current is perpendicular to all harmonic forms. The resulting cycle $z^\Ara =(z, g)$ is the Arakelov lifting of $z$. The admissible height pairing of two disjoint cycles $z_1, z_2$ is then 
$$(z_1, z_2)_\Ara=\int _{X(\BC)} g_1\delta _{z_2}$$
where $z_1^\Ara=(z_1, g_1).$

By Theorem \ref{thm-ssr}, we have the following:
\begin{cor}\label{cor-adm} Assume that  $X/S$ is strictly semistable and that  Grothendieck's standard conjectures
\cite[Conjectures 1, 2]{Gr} holds. Then 
the Arakelov  height pairings are  well-defined for cycles on $X$.
\end{cor}

\subsubsection*{Beilinson--Bloch liftings}
 
 A cycle $z\in \Ch^*(X_\eta)$ is called homologically trivial if its class in $H^{2*}(X_\bata)(*)$ is trivial. 
For a homologically trivial cycle $z\in \Ch ^*(X_\eta)$, a lifting $z^\BB\in \wh \Ch^*(X)$ is called a {\em Beilinson--Bloch lifting} if $z^\BB$ has vanishing curvature in $A^*(X_s)$.

 \begin{prop}\label{prop-bb} Assume either the smoothness of $X/S$  or the standard Conjectures \ref{conj-nd}, \ref{conj-lef}.
An admissible class has curvature $0$  if and only if it is homologically trivial.  In particular, the Beilinson--Bloch lifting exists for all
homologically trivial cycles. 
\end{prop}
\begin{proof}
By Theorem \ref{thm-lef-ac}, the $A_\psi^*(X_s)$ is isomorphic to its image in $H^{2*}(X_{\bata})(*)$.
Thus the curvature map is the following  composition of maps:
$$\ol Z^*(X)\surj A^*(X_\eta)\emb \CA_\psi^*(X_s),$$
where $A^*(X_\eta)$ is the image of $Z^*(X_\eta)$ in $H^{2*}(X_{\bata})(*)$. 
\end{proof} 

As an application of Proposition  \ref{prop-bb},   for two disjoint cycles $z_1\in Z^i(X_\eta)$ and $z_2\in Z^j(X_\eta)$ with $i+j=n+1$ such that $z_1$ is homologically trivial,  we can 
define their {\em Beilinson--Bloch local height  pairing} as follows:
$$(z_1, z_2)_\BB =z_1^\BB\cdot z^\zar_2 .$$

The archimedean analog of the above construction is classical due to Arakelov, Faltings, and Gillet--Soul\'e. For a complex projective variety $X$ with a K\"ahler form $\omega$,  we can extend any homologically trivial cycle $z\in Z^*(X)$ to a Green current $g$ so that $\pp g=\delta _z$.
The Beilinson--Bloch  height pairing of two disjoint cycles $z_1, z_2$ is then  defined as follows:
$$(z_1, z_2)_\BB=\int _{X(\BC)} g_1\delta _{z_2}$$
where $z_1^\BB=(z_1, g_1)$ is an admissible lifting of $z_1$.

For non-strictly semistable reduction, we have the following weaker result:

\begin{thm}\label{thm-bb} 
Assume the Grothendieck standard conjectures \cite[Conjectures 1, 2]{Gr}.  The Beilinson--Bloch lifting exists for every homologically trivial cycle.
\end{thm}

\begin{proof}
 We need to show that for any homologically trivial cycle $z\in \Ch^*(X_\eta)$, we can find an extension $\bar z$ which has vanishing curvature.

 Apply de Jong's theorem \cite {dJ} to get a morphism $\pi: X'\lra X$ such that $f'=f\circ \pi: X'\lra S$ satisfies same property as $f$ with $X'$ having strictly semistable reduction. Let $i': X_s'\lra X$ denote the inclusion of special fiber of $X'$.
 Then $\pi^*z$ is still homologically trivial. Thus there is an extension $\ol{\pi^*z}$ with vanishing curvature.
  Let $\bar z=\pi_! (\ol{\pi^*z})$. Then $\bar z$ also has vanishing curvature. 
\end{proof}

\begin{remark}Using Bloch--Gillet--Soul\'e's harmonic forms,  K\"unnemann  defined an  Arakelov group \cite[\S3.6]{Ku98a},  and related it to the Beilinson--Bloch height pairing \cite[\S3.8]{Ku98a},  
under the assumptions of Theorem \ref{thm-ssr}
with following additional conditions:
\begin{enumerate}
\item   $X/S$ has a model $X_0/S_0$ with $k(s_0)$ a finite field, and
\item  $H^{2*}(Y_I)(*)$ is generated by algebraic classes, and is semisimple under 
$\Gal (k/k(s_0))$ for $|I|=1, 2$ \end{enumerate}
In particular, his work covers the case of abelian varieties over local fields with total degeneration and the case of varieties uniformized by the Drinfeld upper-half spaces. 
\end{remark}

\section{Global  cycles}\label{sec-glb}

In this section,  for arithmetic varieties or algebraic varieties fibered over curves, we  define the Arakelov Chow  groups of admissible cycles
and the decomposition (\ref{thm-ac-dec}) by using 
 our structure theorems for Lefschetz modules in \S\ref{sec-lef} under local and global standard conjectures. 
We will show that  the global standard conjecture for the Arakelov Chow groups   is essentially  equivalent respectively to   the standard conjectures of
Gillet--Soul\'e and Beilinson (\ref{lem-gsadm}, \ref{thm-gsbb}).

We can unconditionally define the Arakelov cohomology groups and the decompositions in the function field case.
These cohomology groups are isomorphic to the intermediate extensions of the cohomology groups over smooth locus. Still, the decomposition (Theorem \ref{thm-coh-dec}) is different than the classical one defined by splitting of cohomology (\ref{eq-coh-dec}). 

For divisors and one cycles, we will give unconditional definitions of  admissible cycles and  decompositions (Theorems \ref{thm-dec-pic}, \ref{thm-dec-0c})  using the Hodge index theorem of Faltings \cite{Fa} and Moriwaki \cite{Mo}.
Thus we obtain an unconditional arithmetic $\sL$-liftings for divisors and $0$-cycles on the generic fibers.
(Corollaries \ref{cor-lef-pic}, \ref{cor-lef-0c}). We obtain some modular generating series of arithmetic Kudla's divisors or $0$-cycles for Shimura varieties of orthogonal or unitary types.

\subsection{Arakelov Chow groups}

\subsubsection*{Arithmetic cycles}
 Let $S$ be a 
regular scheme of dimension $1$, which is either an arithmetic curve 
$S=\Spec \CO_K$ for a number field $K$, or a smooth and projective curve over a field $k$. We call 
a place $v$ of $K$ for a point of $S$ or an infinite valuation in the number field case.

Let $f: X\lra S$ be a projective and flat morphism from a regular scheme of dimension $n+1$.
 We want to define some  modified groups of cycles
$\wh Z^*(X)$  and  $\wh \Ch ^*(X)$ as follows. 

In the geometric situation, we define 
 $\wh Z^*(X)$  as the quotient of $Z^*(X)$ modulo images of homologically trivial cycles on vertical fibers, and   $\wh\Ch^*(X)$  as the image of $\Ch^*(X)$ in $H^{2*}(X_{\bar k})(*)$ for some Weil cohomology. Notice that $H^{2*}(X_{\bar k})(*)$ is a cohomology group for the variety  $X_{\bar k}$
 rather than $H^{2*}(X_{\bar K})(*)$ for its geometric  generic fiber $X_{\bar K}$.
 
 In the arithmetic situation, we define $\wh Z^*(X)$ and $\wh\Ch^*(X)$  to be the quotients of Gillet--Soul\'e's \cite{GS91} groups of arithmetic cycles   $\wh Z^*(X)_{GS}$ and $\wh\Ch^*(X)_{GS}$ modulo the images of homologically trivial cycles on vertical fibers.

Let $L$ be an ample line bundle on $X$. In the arithmetic case, this means that  $L$ is a Hermitian line bundle on $X$ as defined by Gillet--Soul\'e \cite{GS91}  with positive curvature point-wise at archimedean places, and with positive intersections with horizontal cycles, see \cite{Zh}. As usual, let $\sL$ be the Lefschetz operator defined by 
$c_1(L)$. Here  is  (slightly modified) Gillet--Soul\'e's standard conjecture:
\begin{conj}[Gillet--Soul\'e]\label{conj-gs}
Let $i\le (n+1)/2$. 
\begin{enumerate}
\item We have an isomorphism
$$\sL^{n+1-2i}: \wh\Ch^i(X)\iso\wh\Ch ^{n+1-i}(X).$$
\item For  $x\in \wh\Ch^i(X)$, $x\ne 0$, and $\sL^{n+2-2i}x=0$, we have 
$$(-1)^i(x, \sL^{n+1-2i}x)>0.$$
\end{enumerate}
\end{conj}

\subsubsection*{Admissible cycles}

Let $s$ be a place of $K$. If $s$ is a closed point  of $S$, then there is a   morphism of schemes:
$$\breve s:=\Spec \breve \CO_{S, s}\lra  S,$$
where  $\breve \CO_{S, s}$ denotes the completion of a maximal unramified extension of $\CO_{S, s}$. This induces  a morphism
$$f_{\breve s}: X_{\breve s} :=X\times_S {\breve s}\lra \breve s.$$
Then we define 
$$A^*(X_s)=A^*(X_{\bar s})^{\Gal (\bar s/s)}.$$
If $s$ is infinite given by an embedding $K\lra  \BC$, then we have $K_{\bar s}\iso \BC$ and $K_s=\BR$ or $\BC$.
With our notation in Corollary \ref{cor-arch}, we define  
$$A^*(X_s)=A^*(X_{\bar s})^{\Gal (\bar s/s)}.$$

Now we assume  for  each closed point $s$ of $S$ that Conjecture \ref{conj-lef} holds for $f_{\breve s}$, and that Conjecture \ref{conj-nd}
holds 
when  $f_{\breve s}$ is not smooth. Then by Theorem \ref{thm-harac}, and Corollary  \ref{cor-arch},
 there is a harmonic decomposition 
$$A^*(X_{\bar s})=A_\varphi^*(X_{\bar s})\oplus \CA^*_\psi(X_{\bar s}).$$
Taking Galois invariants, we get a decomposition 
$$A^*(X_{s})=A_\varphi^*(X_{s})\oplus \CA^*_\psi(X_{s}).$$

We say that a class  $x\in \wh \Ch^*(X)$ is {\em admissible}  at $s$, if its curvature $\omega_s(x)$ is harmonic. 
We say such a class is {\em admissible} if it is  admissible everywhere.
Following  Gillet--Soul\'e \cite{GS91} and K\"unnemann \cite{Ku98a},  we define the {\em Arakelov 
Chow group} as the group of admissible classes:
$$\ol\Ch^*(X):=\left\{ x\in \wh \Ch^*(X): \omega _s(x)\in \CA_\psi^*(X_{s}), \quad \forall s\right\}.$$
We define the group of vertical cycles and the curvature map by 
$$\omega_\varphi: \quad \wh\Ch^*_\varphi (X):=\sum_s i_{s*}A_{n+1-*}(X_s)\surj A_\varphi ^*(X):=\bigoplus _s A_\varphi ^*(X_s).$$
Denote  the kernel of this curvature map by $B^*(X)$.
Then we have  the following identities and exact sequence:
$$\wh\Ch^*(X)=\ol \Ch^*(X)+\wh\Ch_\varphi^*(X), \qquad B^*(X)=\ol \Ch^*(X)\cap \wh\Ch_\varphi^*(X).$$
$$0\lra B^*(X)\lra \ol\Ch^*(X)\lra \Ch^*(X_K)\lra 0.$$

By Theorem \ref{thm-lef-ac}, the $A_\psi^*(X_s)$ is isomorphic to its  image in $H^{2*}(X_{\bar K})(*)$.
Thus the curvature map is the following  composition
$$\ol\Ch^*(X)\surj A^*(X_K)\emb \CA_\psi^*(X_s),$$
where $A^*(X_K)$ is the image of $\Ch^*(X_K)$ in $H^{2*}(X_{\bar K})(*)$. This implies that an admissible class has curvature zero at one place if and only if it is homologically trivial. Thus we have well-defined Beilinson--Bloch height pairing on the group $\Ch^*(X_K)^0$ of homologically trivial cycles.

\begin{lem}\label{lem-gsadm}
Assume the  standard conjectures 
\ref{conj-lef}, \ref{conj-hod} for $A_\varphi^*(X_s)$, $A_\psi ^*(X_s)$   for every place $s$ of $K$.
Then the  standard conjecture \ref{conj-gs} for $\wh \Ch^*(X)$ is equivalent to the standard conjecture  for $\ol \Ch^*(X)$
of Lefschetz and Hodge types.
 Moreover there is  an orthogonal   decomposition:
$$\wh\Ch^*(X)=\ol \Ch^*(X)\oplus A_\varphi^*(X).$$
\end{lem}
\begin{proof}
There is an exact sequence:
$$
0\lra \ol\Ch^*(X)\lra \wh\Ch^*(X)\lra\ A_\varphi ^*(X)\lra  0.
$$
The truth of the standard conjecture for any two of three terms in the above sequence will imply the truth for the third one. 
\end{proof}

\subsubsection*{Filtrations}

 We consider a 3-step filtration $F^*\ol\Ch^*(X)$ by
$$F^i\ol\Ch^*(X)=\begin{cases} \ol\Ch^*(X)&\text{if $i=0$,}\\
\ol\Ch^*(X)^0&\text{if $i=1$,}\\
B^*(X)&\text{if $i=2$,}
\end{cases}
$$
where $\ol\Ch^*(X)^0:=
\Ker (\ol\Ch^*(X)\lra H^{2*}(X_{\bar K})(*))$.
This filtration has graded quotients  given by 
$$G^i\ol\Ch^*(X)=
\begin{cases}
A^*(X_K)&\text{if $i=0$,}\\
\Ch^*(X_K)^0&\text{if $i=1$,}\\
B^*(X)&\text{if $i=2$,}
\end{cases}
$$
where $A^*(X_K)$ and $\Ch^*(X_K)^0$ are respectively the image and the kernel of the following map 
$$\Ch^*(X_K)\lra H^{2*}(X_{\bar K})(*).$$
Notice that the intersection pairing on $\ol\Ch^*(X)$ induces the Beilinson--Bloch height pairing on $\Ch^*(X_K)^0\iso G^1\ol\Ch^*(X_K)$.
Let $\epsilon \in B^1(X)$ denote a class of degree $1$: 
$\epsilon=\pi^*\epsilon _K$ for $\epsilon_K \in \ol\Ch^1(K)$ with degree $1$.
Then the intersection with $\epsilon$ on $\ol\Ch^*(X)$ factors through $A^*(X_K)$ with the image in $B^*(X)$.

\subsection{Decompositions}
Applying Theorem \ref{thm-fil-dec}  and Proposition \ref{prop-sym},
we have  a canonical splitting of this filtration with respect to operator 
$\sL$ and adjoint $\sLambda$. 

\begin{thm}\label{thm-ac-dec} Assume  standard Conjectures \ref{conj-lef}, \ref{conj-hod}, and \ref{conj-gs}. There is a unique splitting  of filtered  $\BR$ modules
$$\alpha=(\alpha ^0, \alpha ^1, \alpha^2): G^0\ol \Ch^*(X)\oplus G^1\ol \Ch^*(X)\oplus G^2\ol \Ch^*(X)\iso \ol \Ch^*(X)$$ 
  such that $\alpha ^1$ is $\sL$-linear, and $\alpha ^0$ is $\sL$-linear  modulo $\Im \alpha ^2$ and $\sLambda$-linear  modulo $\Im\alpha ^1$.
  Moreover we have the following properties for this splitting:
\begin{enumerate}
\item $\alpha ^1$ is  isometric when $G^1\ol\Ch^*(X)\iso \Ch^*(X_K)^0$ is equipped with the Beilinson--Bloch height pairing;
\item $\Im\alpha ^0$ is isotropic;
\item  there is an $\sL$-linear isomorphism $\beta: G^0\ol \Ch^*(X)\lra G^2\ol\Ch^{*+1}(X)$ such that 
  $\alpha$ translates the $\BR [\sL, \sLambda]$-module structure on $\ol\Ch^*(X)$ to a structure on $\bigoplus _i G^i \ol\Ch^*(X)$  defined as follows: 
for $(x^0, x^1, x^2)\in \bigoplus_{i=0}^2G^i\ol\Ch^*(X)$, 
$$ \sL  \begin{pmatrix} x^0\\ x^1\\ x^2\end{pmatrix}:
=\begin{pmatrix}\sL &0&0\\
0&\sL &0\\
\beta &0&\sL\end{pmatrix}\begin{pmatrix} x^0\\ x^1\\ x^2\end{pmatrix},
\qquad 
 \sLambda\begin{pmatrix} x^0\\ x^1\\ x^2\end{pmatrix}:
=\begin{pmatrix}\sLambda &0& \beta^{-1}\\
0&\sLambda &0\\
0 &0&\sLambda\end{pmatrix}\begin{pmatrix} x^0\\ x^1\\ x^2\end{pmatrix}.$$
\end{enumerate}
\end{thm}

\begin{example}  It is clear that  $\alpha ^0 [X_K]=[X]$, and $\beta [X_K]=c X_\epsilon $ for some $c\in \BR$.
We compute $c$ as follows:
$$c_1(L)^{n+1}=\deg (\sL^{n+1}\alpha ^0  [X])= (n+1)\deg(c_1(L)^n \alpha ^2 \beta [X_K])
=(n+1)c_1(L_K)^n\beta([X_K]).$$
It follows that 
$$\beta([X_K])=\frac {c_1(L)^{n+1}}{(n+1)c_1(L_K)^n}X_\epsilon  =h_{L}(X_K)X_\epsilon .$$
Then  the lifting of $c_1(L_K)^i\in A^i(X_K)$  under $\alpha ^0$ can be defined as  
$$\alpha ^0  \sL^i [X_K]=\sL^i \alpha ^0  [X_K]-i \sL^{i-1} \beta[X_K]= c_1(L)^i-ih_{L} (X)c_1(L)^{i-1} X_\epsilon =c_1(L_0)^i,$$
where $L _0:=L(-h_L(X))$ is the unique rescaling of $L$ such that $c_1(L_0)^{n+1}=0$.
\end{example}

\begin{remark}[Triple products]\label{rem-triple} As one application of Theorem \ref{thm-ac-dec},  we  consider the following symmetric triple product of cycles :
$$\ol\Ch^i(X)\times \ol\Ch^j(X)\times \ol\Ch^k(X) \lra \BR, \qquad i+j+k=n+1, \qquad (x, y, z)_\BT=\wh\deg (x\cdot y\cdot z).$$
When one of $i, j, k$ is zero, this pairing 
is completely determined by the Beilinson--Bloch  height pairing on $\Ch^*(X)^0$ and by the intersection pairing 
on $A^*(X_K)$. 

In the general case, this pairing is more complicated to understand. If we take $x, y, z$ in the graded pieces of these groups with degree $\ell\le m\le n$, we expect  
$x\cdot y\cdot z=0$ if $\ell+m+n\ge 3$. This is clear if  $n=2$,  and is conjectured by Beilinson
if  $\ell=m=n=1$. Also when  $\ell=0$, $m=n=1$ this is essentially the Beilinson--Bloch  height pairing:
$$(x, y, z)_\BT=(x, yz)_\BB.$$
It remains two interesting cases: $(\ell, m, n)=(0, 0, 1)$ or $(0, 0, 0)$. 

If $(\ell, m, n)=(0, 0, 1)$,  by the perfectness  of the Beilinson--Bloch pairing  on $\Ch^*(X)^0$ we obtain a map
$$A^i(X_K)\times A^j(X_K)\lra \Ch^{i+j}(X_K)^0.$$
It is an exciting question to construct this pairing directly. 

If $\ell=m=n=0$, then there is a triple product for $A^*(X_K)$:
$$A^i(X_K)\times A^j(X_K)\times A^k(X_K) \lra \BR, \qquad i+j+k=n+1, \qquad (x, y, z)_\BT=\wh\deg (x\cdot y\cdot z).$$
This pairing has been used when $X_K$ is the  product of two curves in our previous work on Gross--Schoen cycles \cite{Zh10} and
triple product $L$-series \cite{YZZ21}.
  \end{remark}

\subsubsection*{$\sL$-  liftings}
Using the decomposition in Theorem \ref{thm-ac-dec}, we obtain an isomorphism: 
$$(\Im\alpha ^0)^\perp=\Im \alpha ^0+\Im\alpha ^1\iso \Ch^*(X_K).$$
The inverse of this map defines a canonical admissible  lifting called {\em $\sL$- lifting}:
$$\Ch^*(X_K)\lra \wh \Ch^*(X): z\mapsto z^\sL.$$
Then we can define an intersection pairing called {\em $\sL$-  pairing} on $\Ch^*(X_K)$ by
$$(\cdot, \cdot)_\sL: \Ch^i(X_K)\times \Ch^{n+1-i} (X_K)\lra \BR,\qquad (z_1, z_2)_\sL=\deg (z_1^\sL\cdot z_2^\sL).$$

This pairing has a close relation with Beilinson--Bloch's height pairing. If we assume the standard conjecture for $\Ch^*(X_K)^0$ and $A^*(X_K)$,
then there is a unique splitting of $\sL$-modules:
$$\Ch^*(X_K)\iso A^*(X_K)\oplus \Ch^*(X_K)^0, \qquad z\mapsto (z^\cl, z^0).$$
Thus, 
$$(z_1, z_2)_\sL=(z_1^0, z_2^0)_\BB.$$
This identity follows that $\Im \alpha ^0$ is isotropic and perpendicular to $\Im \alpha ^1$.

\subsection{Arithmetic standard conjectures}
In the following, we  want to compare three arithmetic standard conjectures: Gillet--Soul\'e's standard 
Conjecture \ref{conj-gs}, the standard conjecture on its subgroup $\ol\Ch^*(X)$ of admissible cycles, and the following 
conjecture by Beilinson on the group $\Ch^*(X_K)^0$ of homologically trivial cycles:
\begin{conj}[Beilonson]\label{conj-bb}
Let $i\le (n+1)/2$. 
\begin{enumerate}
\item We have an isomorphism
$$\sL^{n+1-2i}: \Ch^i(X_K)^0\iso \Ch ^{n+1-i}(X_K)^0.$$
\item For  $x\in \Ch^i(X_K)^0$, $x\ne 0$, and $\sL^{n+2-2i}x=0$, we have 
$$(-1)^i(x, \sL^{n+1-2i}x)>0.$$
\end{enumerate}
\end{conj}

One of our main results in this section  is the following:
\begin{thm}\label{thm-gsbb}
Assume  Grothendieck's  standard conjectures \cite[Conjectures 1, 2]{Gr}   for $A^*(X_K)$,  the  standard conjectures 
\ref{conj-lef} and \ref{conj-hod} for $A_\varphi^*(X_s)$, $A_\psi ^*(X_s)$ for every place $s$ of $K$, and the perfectness of the pairing
$$B^*(X)\times A^{n+1-*}(X)\lra \BR.$$
 Then the  following statements hold:
 \begin{enumerate}
 \item 
 The  standard conjecture \ref{conj-gs} for $\wh \Ch^*(X)$ is equivalent to the standard conjecture for $\ol\Ch^*(X)$;
 \item The standard conjecture for $\ol\Ch^*(X)$ implies  Beilinson's standard conjecture \ref{conj-bb}.
 \item Beilinson's standard conjecture \ref{conj-bb} implies the standard conjecture \ref{conj-gs} for any polarization of the form
 $L\otimes \pi^*(c)$, where     $c\in \wh \Pic (\CO_K)$ with  $\deg c$ sufficiently large.
 \end{enumerate}
 \end{thm}
 \begin{proof}
 The part one is implies by Lemma \ref{lem-gsadm}.
 For the other two parts, we apply Theorem \ref{thm-vg1}.
 \end{proof}

\subsection{Cohomology cycles}
In the following, we want to consider the group $H^*(X)$ of cohomology cycles in the function field case. For simplicity, we assume that $k=\bar k$. 
Then for each closed point $s$ of $S$, we have maps:
\begin{equation}\label{eq-ihi}
H_{X_s}^*(X)\overset {i_{s*} } \lra H^*(X)\overset {i_s^*} \lra H^*(X_s).
\end{equation}
Using Corollary \ref{cor-bbdg} and Theorem \ref{thm-harcoh}, we  get  a unique  decomposition as $\BQ[\sL]$-modules:
\begin{equation}\label{eq-pp}
H^*(X_s)=H_\varphi ^*(X_s)\oplus \CH _\psi ^*(X_s).
\end{equation}
Notice that $H_\varphi^*(X_s)\ne 0$ only if $X_s$ is singular. 
We define  the group $\ol H^*(X)$ of {\em admissible cohomological group}  $\ol H^*(X)$ as the class with harmonic curvatures in $\CH_\psi^*(X_s)$ for all $s$:
$$\ol H^*(X):=\left\{\alpha \in H^*(X): i_s^*\alpha \in \CH_\psi^*(X_s), \forall s\in S\right\}.$$
Then we get a decomposition 
$$H^*(X)=H_\varphi^*(X)+ \ol H^*(X), \qquad H_\varphi^*(X):=\sum _{s\in S} i_{s*} H_{X_s}^*(X_s).$$
We define a tree-step filtration on $\ol H^*(X)$ by 
\begin{equation}\label{eq-oH-fil}
F^i\ol H^*(X)=\begin{cases} \ol H^*(X)&\text{if $i=0$}\\
\Ker (\ol H^*(X)\lra H^*(X_{\bar K}))&\text{if $i=1$}\\
\Ker (\ol H^*(X)\lra H^*(X_K))&\text{if $i=2$}
\end{cases}
\end{equation}
where $H^*(X_K):=\varinjlim_U H^*(X_U)$ where $U$ runs through non-empty open subsets of $S$.
Let  $G^i\ol H^*(X)$ ($i=0,1,2$) denote the graded pieces. 
\begin{prop}\label{prop-lef-olh}
The  $\ol H^*(X)$ has a Lefschetz module structure with center $n+1$,
and each $G^i\ol H^*(X)$  has a  Lefschetz module structure with center $n+i$.
More precisely, let $j: U\emb S$ be any non-empty open subscheme of $S$ over which $f$ is smooth.
Then we have the following isomorphisms  $\BQ_\ell[\sL]$-modules,
$$G^i\ol H^*(X)\iso H^i(S, j_*R^{*-i}f_{U*}\BQ_\ell),$$
where the action of  $\sL$ on $H^i(S, j_*R^{*-i}f_{U*}\BQ_\ell)$ is induced from its action on the 
 sheaves $R^if_{U*}\BQ_\ell$.
\end{prop}
Applying Theorem \ref{thm-fil-dec}, we will get a  decomposition analogous to Theorem \ref{thm-ac-dec}:
\begin{thm}\label{thm-coh-dec} There is a unique splitting  of filtered  $\BQ_\ell$ modules
$$\alpha=(\alpha ^0, \alpha ^1, \alpha^2): G^0\ol H^*(X)\oplus G^1\ol H^*(X)\oplus G^2\ol H^*(X)\iso \ol H^*(X)$$ 
  such that $\alpha ^1$ is $\sL$-linear, and $\alpha ^0$ is $\sL$-linear  modulo $\Im \alpha ^2$ and $\sLambda$-linear  modulo $\Im\alpha ^1$.
\end{thm}

To prove Proposition \ref{prop-lef-olh},  we need to reinterpret the cohomology $\ol H^*(X)$ and its filtration 
 in terms of decomposition theorems for the complex $Rf_*\BQ_\ell$ on $S$ in  \cite[Theorem 5.4.5, 5.4.6]{BBDG} in characteristic $p$ and 
 \cite[Theorem 5.4.5, 5.4.6]{BBDG} in characteristic $0$. More precisely, we will compare $Rf_*\BQ_\ell$ with  intermediate complex 
 $\ol Rf_*\BQ_\ell:=j_{!*} R_{U*}\BQ_\ell$ which has  cohomology
 $\ol R^if_*\BQ_\ell =j_*R^i_{U*}\BQ_\ell$.  Our first step is to write down some  decompositions. 
 
 First of all, the analog Lemma \ref{lem-dec-pH} holds for sheaves  $D_c^b(S)$ with the same proof.
 Thus there is a (non-canonical) decomposition:
 \begin{equation}\label{eq-bbdg0}
 Rf_*\BQ_\ell \simeq \bigoplus _{m\in \BZ} \pR^mf_*\BQ_\ell K[-m].
 \end{equation}

 Secondly, the global analogue of the  Conjecture \ref {conj-perv-dec} holds for over $S$:  there is a  splitting of complexes:
\begin{equation}\label{eq-bbdg1}
\pR^if_*\BQ_\ell \iso \CH^0(\pR^if_*\BQ_\ell )\oplus \CH^{-1}(\pR^if_*\BQ_\ell)[1],
\end{equation}
\begin{equation}\label{eq-bbdg2}  \CH^{-1}(\pR^if_*\BQ_\ell)\iso j_*R^{i-1}f_{U*}\BQ_\ell=\ol R^{i-1}f_*\BQ_\ell.
\end{equation}

 Finally, the global analogue of Conjecture  \ref{conj-perv-lef} holds:
for any $i\le n+1$, we have an isomorphism
\begin{equation}\label{eq-bbdg3}\sL^i : \pR^{n+1-i} f_*\BQ_\ell \iso  \pR^{n+1+i}f_*\BQ_\ell (i).
\end{equation}

Now we want to translate these isomorphisms in terms of usual cohomology:
The first  isomorphisms (\ref{eq-bbdg0}), (\ref{eq-bbdg1}), and (\ref {eq-bbdg2})
gives a single  isomorphism:
 $$Rf_*\BQ_\ell= \bigoplus_m  (\Phi^m\oplus \ol R^mf_{*}\BQ_\ell)[-m],\qquad \Phi^m:=\Ker (R^mf_*\BQ_\ell\lra  \ol R^mf_{*}\BQ_\ell).$$
It is clear that each $\Phi^m$ is a complex of sheaves supported on $S\setminus U$.
The last isomorphism (\ref{eq-bbdg3}) gives two isomorphisms:
$$\sL^i : \Phi^{n+1-i}\iso  \Phi^{n+1+i} (i), \qquad \sL^i : \ol R^{n-i} f_{U}*\BQ_\ell \iso  \ol R^{n+i}f_{*}\BQ_\ell (i).$$
Applying Proposition \ref{prop-lef-dec} to the exact sequence 
$$0\lra \Phi^*\lra R^*f_*\BQ_\ell\lra \ol R^*f_{*}\BQ_\ell\lra 0$$
we obtain a unique decomposition of $\BQ_\ell[\sL]$ modules:
$$
R^*f_*\BQ_\ell\iso \Phi^*\oplus \ol R^*f_{*}\BQ_\ell.
$$
Thus we have proved the following:
\begin{lem} \label{lem-dec-rphi}There is a unique decomposition  of the form
(\ref{eq-bbdg0}) respecting the action by Lefschetz operator $\sL$.
More precisely, there are  canonical splittings of $\BQ_\ell[\sL]$ modules:
$$Rf_*\BQ_\ell\iso \Phi^*\oplus \ol Rf_*\BQ_\ell,$$
$$\ol Rf_*\BQ_\ell\iso \bigoplus_m  \ol R^mf_{*}\BQ_\ell [-m].$$
\end{lem}
At each closed point $s\in S$,  Lemma  \ref{lem-dec-rphi}  gives a splitting $\BQ[\sL]$ modules
$$H^*(X_s)=\Phi_s^*\oplus (\ol R^*f_{*}\BQ_\ell)_s.$$
This must be coincides with decomposition (\ref{eq-pp}). Thus we have 
$$\Phi_s^*=H_\varphi^*(X_s), \qquad (\ol R^*f_{*}\BQ_\ell)_s=\CH_\psi^*(X_s).$$
Over $S$, Lemma \ref{lem-dec-rphi} implies the following identities:
$$H^*(X)=H^*(S, Rf_*\BQ_\ell), \qquad H_\varphi^*(X)=H^*(S, \Phi^*), \qquad \ol H^*(X)= H^*(S, \ol Rf_*\BQ_\ell).$$
In particular, there is  a splitting  of cohomology:
\begin{equation}\label{eq-coh-dec}
\ol H^*(X)=H^*(S,  \ol Rf_*\BQ_\ell )=\bigoplus_{i=0}^2 H^i(S, \ol R^{*-i}f_*\BQ_\ell).
\end{equation}
This splitting is compatible  with the   filtration $F^i\ol H^*(X)$ defined as in (\ref{eq-oH-fil}). 
Thus we have $G^i\ol H^*(X)=H^i(S, \ol R^{*-i}f_*\BQ_\ell)$ which are Lefschetz modules with center $n+i$.
This completes the proof of the proposition. 

\begin{remark} Notice that the decomposition in Theorem \ref{thm-coh-dec} depends on the Lefschetz operator $\sL$
while the decomposition in (\ref{eq-coh-dec}) does not. More precisely module the Tate twists, 
write $\sL=\sL_\varphi+\sL_0+\sL_1+\sL_2$
with respect to the decomposition
$$H^2(X)=H_\varphi^2(X)\oplus \bigoplus_{i=0}^2 H^i(S, \ol R^{*-i}f_{*}\BQ_\ell).$$
Then on the decomposition (\ref{eq-coh-dec}), $\ol H^*(X)$, $\sL_\varphi$ acts trivially , $\sL_j$ acts by its action on $ \ol R^{*-i}f_{*}\BQ_\ell$,
$$\sL_j H^i(S, \ol R^{*-i}f_*\BQ_\ell)\subset H^{i+j}(S, \ol R^{*-i-j}f_*\BQ_\ell), \qquad j=0,1,2.$$
This shows that the two decompositions in Theorem \ref{thm-coh-dec} and (\ref{eq-coh-dec})
are different if  $\sL_1\ne 0$.
\end{remark}
\begin{remark}
As in Remark \ref{rem-triple}, there is a triple product on groups $\ol H^*(X)$. It would be interesting to construct such a triple pairing directly.
\end{remark}

\begin{remark}[$\ell$-adic height pairings]
Recall from the last section,  the group $\wh\Ch^*(X)$ of ``arithmetic Chow cycles" is defined as  the image of
$\Ch^*(X)$ in $H^{2*}(X)(*)$. If we assume local standard conjectures, then we can define the subgroup $\ol\Ch^*(X)$ 
of admissible classes which is in fact the intersection $\wh\Ch^*(X)\cap \ol H^{2*}(X)(*)$. Furthermore the filtrations on 
$\ol\Ch^*(X)$ and on $\ol H^{2*}(X)(*)$ are compatible. Thus two   Theorems \ref{thm-ac-dec}
and \ref{thm-coh-dec} give the same decomposition for $\ol\Ch^*(X)$. 

Without local standard conjectures, we can use  Theorem \ref{thm-coh-dec} to lift cycles in $A^*(X_K)$ and $\Ch^*(X_K)^0$ to $\ol H^{2*}(X)(*)$
by embeddings 
$$A^*(X_K)\incl G^0\ol H^{2*}(X)(*), \qquad \Ch^*(X_K)^0\incl G^1\ol H^{2*}(X)(*).$$
 In particular, there is a  well-defined $\ell$-adic height pairing on $\Ch^*(X_K)^0$ which  has been   defined by Beilinson \cite{Be}.
\end{remark}

\subsection{Divisors and $0$-cycles}
\subsubsection*{Local decomposition}
Let $s$ be a place of $K$ and $i=n$ or $1$.
 We want to define unconditionally a splitting for the group $A_i (X_s)$  of $i$-vertical cycles.
 We start with an intersection pairing and a power of Lefschetz operator:
$$(x, y): A_1(X_s)\otimes A_n(X_s)\lra \BQ,
\qquad (x, y)=\deg (i^*i_*x\cup y).$$
$$\sL^{n-1}: A_n (X_s)\lra A_1(X_s).$$
The following is the classical local index theorem:
\begin{lem}\label{lem-div} 
 For $x\in A_1 (X_s) $ we have 
 $$(x, \sL ^{n-1}x) \le 0.$$
 The equality holds if and only if $x\in \BQ\cdot [X_s]$.
\end{lem}

\begin{cor}\label{cor-sp}
Let $i=1$ or $n$. There is a decomposition
$$
A_i(X_s)=A_i^\psi (X_s)\oplus \CA^\varphi_i (X_s).
$$
Here $A_i^\psi (X_s)$ and $\CA^\varphi_i(X_s)$ are defined as follows:
\begin{enumerate}
\item If  $i=n$, we take 
$$A_n^\psi (X_s): =\BQ\cdot [X_s], \qquad \CA^\varphi _n (X_s):=\left\{x\in A_n(X_s), \deg \sL ^{n}x=0\right\}.$$
\item  If $i=1$, we take
$$A_1^\psi (X_s): =\left\{x\in A_1(X_s): (x, y)=0\quad  \forall y\in A_n(X_s)\right\},
\qquad  \CA^\varphi _1 (X_s):=\sL ^{n-1}\CA^\varphi_n.$$
\end{enumerate}
\end{cor}

\subsubsection*{Admissible cycles}
For $i=1, n$, we want to define the Arakelov Chow group  $\ol\Ch^i(X)$ of admissible cycles.
First we want to do some analysis on vertical cycles for each place $s$ of $S$.
Let $i_s: X_s\lra X$ be the embedding. Then there are maps
$$i_{s*}:  A_{n+1-i}(X_s)\lra \wh \Ch^i(X).$$
We define  $A_\varphi ^i(X)$ (resp $B^i(X)$ )
as the sum of $i_{s*}\CA_{n+1-i}^\varphi(X_s)$ (resp. $i_{s*}A_{n+1-i}^\psi(X_s)$ ), 
and 
 $\ol \Ch^i(X)$ to be the orthogonal complements of 
$A_\varphi ^{n+1-i} (X)$. Then there is a decomposition and an exact sequence:
$$\wh \Ch^i (X)=A_\varphi^i(X)\oplus \ol\Ch^i(X),$$
$$0\lra B^i(X)\lra \ol \Ch^i(X)\lra \Ch^i(X)\lra 0.$$
In terms of curvatures, the subgroup  $\ol\Ch^i(X)$ of $\wh \Ch^i(X)$ consists of elements $x$
such that the volume form $\sL^{n-i} x$ (as a functional over $\sum _si_*\Ch_n (X_s) $) is proportional to 
$c_1( L)^n$. 

Again, there is an intersection pairing
$$\ol \Ch^1(X)\otimes \ol \Ch^n(X)\lra \BR.$$
Let $C^i(X)$ be the null space of this pairing. Also there is a power of  Lefschetz operator
$$\sL^{n-1}: \wh \Ch^1(X)\lra \wh\Ch^n (X).$$
We have the following Hodge index theorem deduced from the local index theorem \ref{lem-div}
and the global Hodge index theorem of  Faltings  \cite{Fa},  Hriljac and  Moriwaki \cite{Mo}.

\begin{thm}[Hodge index theorem]\label{thm-glohod}
For any non-zero $x\in \ol \Ch^1(X)$ with $\sL^n x=0$, then 
$$(x, \sL^{n-1} x)<0.$$
\end{thm}

\begin{cor}\label{cor-0-cycle} The morphism $\sL^{n-1}$ is injective, $C^1(X)=0$, and 
$$\ol \Ch^n(X)=C^n (X)\oplus \sL^{n-1}\ol \Ch^1(X).$$
\end{cor}

\begin{remark}
It is conjectured that $C^n(X)=0$ or equivalently, 
$\ol\Ch^n(X)=\sL^{n-1}\ol\Ch^1(X)$. In fact, by the non-degeneracy of the Neron--Tate height pairing, 
$C^n(X)$ is isomorphic to the kernel of the Abel--Jacobi map $\Ch^n(X_K)^0\lra \Alb (X)_\BR$.
\end{remark}

\subsubsection*{Decomposition of arithmetic diviors}

Consider a 3-step filtration  for Arakelov group of divisors $\ol\Ch^1(X)$ by
$$F^i\ol\Ch^1(X)=\begin{cases} \ol\Ch^1(X),&\text{if $i=0$,}\\
\ol\Ch^1(X)^0,&\text{if $i=1$,}\\
B^1(X),&\text{if $i=2$.}
\end{cases}
$$
where $\ol\Ch^1(X)^0=\Ker (\ol\Ch^1(X)\lra H^{2} (X_{\bar K})(1))$. This filtration has graded quotients given by 
$$G^i\ol\Ch^1(X)=
\begin{cases}
A^1(X_K),&\text{if $i=0$,}\\
\Ch^1(X_K)^0,&\text{ if $i=1$,}\\
B^1(X),&\text{if $i=2$,}
\end{cases}
$$
where $$A^1(X_K)=\NS (X_K)\otimes \BR,\qquad \Ch^1(X_K)=\Pic (X_K)\otimes \BR, \qquad \Ch^1(X_K)^0=\Pic^0 (X_K)\otimes \BR.$$

\begin{thm}\label{thm-dec-pic}There is a unique splitting of the filtration on $\ol\Ch^1(X)$: 
$$\alpha=(\alpha ^0, \alpha ^1, \alpha ^2): G^0\ol\Ch^1(X)\oplus G^1\ol\Ch^1(X)\oplus G^2\ol \Ch^1(X)\iso \ol \Ch ^1 (X).$$
with following properties:
\begin{itemize}
\item $\alpha ^2$ is the  embedding of $B^1(X)$;
\item 
$\alpha^1$ is the  lifting  $\xi\in \Ch^1(X_K)^0$ to  a  class $\alpha ^1 \xi\in F^1\ol \Ch^1(X)$ such that $\deg c_1(L) ^n \alpha ^1\xi=0$;
\item  $\alpha ^0$ is the lifting which takes  a class $\xi\in A^1(X_K)$ to a class $\alpha ^0(\xi)\in \ol \Ch^1(X)$ with following two properties:
\begin{enumerate}
\item on $X_K$, $c_1(L_K) ^{n-1} \alpha ^0\xi_K\in \Ch^n(X_K)$ is proportional to $c_1(L_K)^n$;
\item the intersection number $c_1(L_0)^n \alpha ^0\xi=0$, where $L_0=L-h_L (X_K) X_\epsilon$.
\end{enumerate}
\end{itemize}
\end{thm}

\begin{cor}[$\sL$- liftings for  divisors]\label{cor-lef-pic}
For any line bundle $M$ there is a unique  arithmetic line bundle $M^\sL$ extending $M$ with the following two conditions:
\begin{enumerate}
\item $M^\sL$ is admissible in the sense that the volume form  $c_1(M^\an)\cdot c_1(L)^{n-1}$ is proportional to $c_1(L)^n$
at every place $s$ of $K$.
\item $\deg \wh c_1(L_0)^n \wh c_1(M^\sL)=0$.
\end{enumerate}
\end{cor}

\begin{remark}
In the case $X=C\times C$ a self-product of a curve and $L=p_1^*\xi+p_2^*\xi$ for an ample line bundle $\xi$.
Such a lifting has been constructed using adelic line bundles and used to study heights of Gross--Schoen cycles and 
triple product $L$-series, see \cite{Zh10, YZZ21}.
\end{remark}

\subsubsection*{Decomposition of $1$-cycles}
In the following we want to give a decomposition of $1$-cycles for a modified group
$$\ol\Ch^n(X)'=\ol\Ch^n(X)/(B^n(X)\cap C^n(X)), \qquad B^n(X)'=B^n(X)/(B^n(X)\cap C^n(X)).$$
Recall that $B^n(X)\cap C^n(X)$ is the null space for the pairing $B^n(X)\times A^1(X_K)\lra \BR$.
Then we still have an exact sequence
$$0\lra B^n(X)'\lra \ol\Ch^n(X)'\lra \Ch^n(X_K)\lra 0.$$

Consider a 3-step filtration for Arakelov group of $1$-cycles $\ol\Ch^n(X)$ by
$$F^i\ol\Ch^n(X)'=\begin{cases} \ol\Ch^n(X)',&\text{if $i=0$,}\\
\ol \Ch^n(X)^0,&\text{if $i=1$,}\\
B^n(X),&\text{if $i=2$,}
\end{cases}
$$
where $\ol\Ch^n(X)^0:=\Ker (\ol\Ch^n(X)'\lra H^{2n} (X_{\bar K})(n))$.
This filtration has graded quotients given by 
$$G^i\ol\Ch^n(X)'=
\begin{cases}
A^n(X_K),&\text{if $i=0$,}\\
\Ch^n(X_K)^0,&\text{ if $i=1$,}\\
B^n(X)',&\text{if $i=2$,}
\end{cases}
$$
where $$ A^n (X_K)\iso \BR \cdot c_1(L_K)^n,\qquad \Ch^n (X_K)^0=\ker \deg.$$
\begin{thm}\label{thm-dec-0c} There is a unique splitting of the filtration on $\ol\Ch^n(X)$ 
$$\alpha=(\alpha ^0, \alpha ^1, \alpha ^2): G^0\ol\Ch^n(X)'\oplus G^1\ol\Ch^n(X)'\oplus G^2\ol \Ch^n(X)'\iso \ol \Ch ^n(X)'.$$
with the following properties:
\begin{enumerate}
\item   $\alpha ^2$ is  the given embedding of $B^n(X)'$, 
\item $\alpha ^1$ is the unique lifting so that $\Im \alpha ^1$ is perpendicular to $\alpha ^0(A^1(X_K))$,
\item $\alpha ^0$ on $A^n (X_K)\simeq \BQ\cdot \sL ^n [X_s]$ to  take $\sL ^n [X_s]$ to $c_1(\ol L_0)^n$. 
\end{enumerate}
\end{thm}

\begin{cor}[$\sL$- liftings for $0$-cycles] \label{cor-lef-0c} For any $\xi\in \Ch^n(X_K)$, there is a  unique lifting $\xi^\sL\in \wh \Ch^n (X)$ 
modulo $B^n(X)\cap C^n(X)$
with the following two conditions:
\begin{enumerate}
\item $\xi^\sL$ is admissible in the sense that its curvature form  is proportional to $c_1(L)^n$ at every place $s$ of $K$.
\item $\deg  \alpha ^0(x)\xi^\sL=0$ for all $x\in \NS (X_K)$.
\end{enumerate}
\end{cor}

\subsubsection*{Modularity of arithmetic Kudla's generating series}
Consider a Shimura variety $X$ defined by an orthogonal (resp. hermitian space)  $V$ over a totally real (resp.  CM) field 
 $F$ of signature 
$(n, 2), (n+2, 0),\cdots (n+2, 0)$ (resp. $(n, 1), (n+1, 0), \cdots (n+1, 0)$.
Then there is a projective system of Shimura varieties $X_U$ of dimension $n$ over $F$ indexed  
by  open and compact subgroups of $\GSpin (\wh V)$ ($\GU (\wh V)$).
For each integer between $0$ and $n$ and each Bruhat--Schwartz function $\phi\in \CS(\wh V^r)$ 
there are generating series of Kudla cycles $Z_\phi$ of codimension $r$ special cycles, see
\cite{Kud, YZZ, Liu}. 

In \cite{Kud}, Kudla conjectured that these series 
of spaces are modular for symplectic groups $\GSp_r$ (resp. unitary group $U(r, r)$) over $F$.
 In his thesis, Wei Zhang proved such modularity 
under the condition that these series are convergent, see also extensions in \cite{YZZ, Liu}.
For unconditional modularity, there are the case of divisors in \cite{YZZ, Liu}, and the case $F=\BQ$ by work of 
Bruinier and Westerholt-Raum \cite{BW}.

Let $L$ be an arithmetic ample line bundle over an integral model of $X$. Then using our conditional $\sL$- lifting, 
we get generating series $Z_\varphi^\sL$. As the lifting is unique, the modularity of $Z_\varphi$ will imply the modularity of $Z_\varphi^\sL$.
In particular, when $r=1$ or $F=\BQ$ we have a modular generating series of arithmetic cycles $Z_\varphi^\sL$.

\appendix
\section{Lefschetz modules}\label{sec-lef}

In this appendix, we prove some results about splittings of    Lefschetz modules with filtrations of  2 or 3 steps. We will start with an easy result about two-step filtrations with centers $(n+1)/2$ and $n/2$ for their graded quotients,
which will be used almost everywhere in the paper. Then we will study the significantly more complicated case of filtrations with three steps that will only be used in \S\ref{sec-glb} for global cycles. 

Let $E$ be a field of  characteristic $0$. We will consider the abelian category $\BM$  of graded vector spaces $V^*=\oplus_{i\in \BZ}  V^i$ over $E$  with a linear operator $\sL$ of degree $1$: $\sL: V^*\lra V^{*+1}$ such that $V^i=0$ if $|i|>>0$.

By a {\em Lefschetz module} with center $n/2$, for $n$ a non-negative integer, we mean an object $V^*\in \BM$ 
such that for any integer $i\le n/2$, there is an isomorphism:
$$\sL^{n-2i}: V^i\iso V^{n-i}.$$

Define the {\em primitive part} by 
$$V_0^i:=\Ker (\sL ^{n+1-2i}|V^i), \qquad i\le n/2.$$
Then there is a Lefschetz decomposition:
  $$V^i=\sum _{j\le\min (n/2. i)} \sL ^{i-j }V_0^j.$$
  It is well known that for any Lefschetz structure $V^*$  with center $n/2$, there is a unique operator $\sLambda$ on $V^*$ with degree $-1$ such  that $[\sLambda, \sL]|V^i=n-2i$. 
  
  For the uniqueness of $\sLambda$, we notice that any two such $\sLambda$'s will have a difference operator 
  $\Delta$ with degree $-1$ and commuting with $\sL$. It follows that for any $j\le n/2$, 
  $$\sL^{n+1-j}\Delta V^j_0=\Delta \sL^{n+1-j}V^j_0=0.$$
  Thus $\Delta V^j_0=0$ since $\sL^{n+1-j}$ is bijective on $V^{j-1}$. Thus $\Delta =0$ on $V$ as it commutes with $\sL$.
  
  For the existence of $\sLambda$, we notice that the above primitive  decomposition realize $V$ as a direct sum of simple modules
   of forms $V(n, j):=\bigoplus_{i=j}^{n-j}E e_i$ with $j\le n/2$ and 
   $$\sL e_i=(n-j-i)e_{i+1}, \qquad \sLambda e_i=(i-j) e_{i-1},$$  
   where we treat $e_{j-1}=e_{n-j+1}=0$.

\subsection{Two-step filtrations}

\begin{prop}\label{prop-lef-dec}  Let $U^*$ and $W^*$ be 
Lefschetz modules with centers $m/2$ and $n/2$ respectively. 
Then in the category $\BM$, the following hold:
\begin{enumerate}
\item if $m>n$, then $\Hom_\BM (W, U)=0$;
\item if $m=n+1$, then $\Ext ^1_\BM(W, U)=0$;
\end{enumerate} 
\end{prop}
\begin{proof}  For the first one, let $\varphi\in \Hom_\BM (W, U)$.
Then for any $i\le n/2<m/2$, 
$$\sL^{m-2i}\varphi (W^i_0)=\sL^{m-n-1}\varphi (\sL^{n+1-2i}W^i_0)=0.$$
This implies that $\varphi (W^i_0)=0$. Thus $\varphi=0$.

For the second part, consider an extension $V$ of $W$ by $U$ in $\BM$:
$$0\lra U\lra V\lra W\lra 0.$$
Since $\sL^{n+1-2k}: U^k\lra U^{n+1-k}$ is an isomorphism,  there is a canonical lifting of primitive parts:
$$V^i_0\iso W^i_0, \qquad V_0^i:=\Ker (\sL^{n+1-2k}|V^i), \qquad i\le n/2.$$
Thus we can define a lifting of $W^i$ by
$$W^i_\sL:=\sum _{k\le\min (n/2, i)}\sL ^{i-k} V ^k_0.$$
The uniqueness of this lifting follows from part 1. 
\end{proof}

 \subsection{Three-step filtrations}

\begin{thm}\label{thm-fil-dec}
 Let $V^*$ be a Lefschetz module over $E$ with center $(n+1)/2$ with a three step  filtration in objects in $\BM$:
$$0\emb F^2V^*\emb F^1 V^*\emb F^0V^*=V^*$$
such that their graded quotients  $G^iV:=V^i/V^{i+1}$ are  Lefschetz modules with  center $(n+i)/2$.
Then there is a unique splitting of graded  $E$-modules
$$\alpha: G^0V^*\oplus G^1V^*\oplus G^2V^*\iso V^*$$ 
  such that the following hold for the restrictions $\alpha^i=\alpha |G^iV^*: G^iV^*\lra F^iV^*$:
\begin{enumerate}
\item $\alpha^1:  G^1V^*\lra V^*$ is  $\sL$-linear;
\item $\alpha^0: G^0V^*\lra V^*$ is $\sL$-linear  modulo $\Im \alpha ^2$ and $\sLambda$-linear  modulo $\Im\alpha ^1$.
\end{enumerate}
Moreover there is an  isomorphism $\beta: G^0V^*\lra G^2V^{*+1}$ of $\sL$-modules such that $\alpha$-translates the $E[\sL, \sLambda]$-module structure on $V^*$ into an  $E[\sL, \sLambda]$-module  structure on $\bigoplus G^iV^*$ as follows: for $x^i\in G^iV^*$, 
$$ \sL  \begin{pmatrix} x^0\\ x^1\\ x^2\end{pmatrix}
=\begin{pmatrix}\sL &0&0\\
0&\sL &0\\
\beta &0&\sL\end{pmatrix}\begin{pmatrix} x^0\\ x^1\\ x^2\end{pmatrix},
\qquad 
 \sLambda\begin{pmatrix} x^0\\ x^1\\ x^2\end{pmatrix}
=\begin{pmatrix}\sLambda &0&\beta^{-1}\\
0&\sLambda &0\\
0 &0&\sLambda\end{pmatrix}\begin{pmatrix} x^0\\ x^1\\ x^2\end{pmatrix}.$$
\end{thm}

First, let us reduce the proof to the case $G^1V^*=0$.
Apply Proposition \ref{prop-lef-dec} to the exact sequences
$$0\lra F^2V^*\lra F^1V^*\lra G^1V^*\lra 0, $$
$$ 0\lra G^1V^*\lra V^*/F^2V^*\lra G^0V^*\lra 0,$$
to obtain  unique $\sL$-linear liftings
$$\alpha ^1: G^1V^*\lra F^1V^*, \qquad v: G^0V^*\lra V^*/F^2V^*.$$
Let $\wt V^*$ be the preimage of $v(G^0V^*)$ under the projection $V^*\lra V^*/F^2V^*$. Then there is a direct sum of $\sL$-modules:
$$V^*=\alpha ^1(G^1V^*)\oplus \wt V^*.$$
Since both $V^*$ and $G^1V^*$ have Lefschetz structures with center $(n+1)/2$, so does $\wt V^*$. 
Thus, we reduce the proof of Theorem \ref{thm-fil-dec} to the case $G^1V^*=0$.
In this case, we rewrite the theorem as follows:

\begin{prop} \label{prop-Lam-dec}
Let 
$0\lra U^*\overset \epsilon \lra V^*\overset\eta\lra W^*\lra 0$
 be an exact sequence of graded vector spaces over $E$ with an action by $\sL$ of degree $1$. Assume that for some integer $n$, $\sL$ induces  Lefschetz structures on $U^*$, $V^*$, and $W^*$ with  centers of symmetry  $(n+2)/2$, $(n+1)/2$, and $n/2$  respectively.  Then there is a unique section $\alpha: W^*\lra V^*$ for $\eta$ such that $\alpha$ is  $\sLambda$-linear.
Moreover, there is a unique  $\sL$-isomorphism  $\beta$ of $W^*\lra U^{*+1}$ such that  the operators $\sL$ and $\sLambda$ on $V^*$ are given by 
$$\sL (\epsilon x+\alpha  y)=\epsilon (\sL x+\beta y)+\alpha   \sL y, \qquad \sLambda (\epsilon x+\alpha   y)=\epsilon \sLambda x+\alpha   (\beta^{-1} x+\sLambda y).$$
\end{prop}

The proof of this proposition uses several lemmas.

\begin{lem}\label{lem-ab} Under the assumption of Proposition \ref{prop-Lam-dec}, there is a  lifting 
  $\alpha : W^*\lra V^*$, and an $\sL$-isomorphism $\beta: W^*\lra U^{*+1}$  satisfying  the following conditions:
$$\sL \alpha =\alpha  \sL+\epsilon   \beta.$$
\end{lem}

\begin{proof}
 Let us consider the primitive class of $V^i_0$ for $i\le (n+1)/2$, i.e., the kernel of  $\sL^{n+2-2i}$ on $V^i$.
Since this operator gives an isomorphism
 $$\sL^{n+2-2i}: U^i\iso U^{n+2-i}, $$
 there is  an isomorphism:
 $$\eta: V^i_0\iso \Ker (\sL^{n+2-2i}: W^i\lra W^{n+2-i})= \sL W^{i-1}_0+W^i_0.$$
 This defines a lifting $\alpha_i: W^i_0\emb V^i_0$. 
 
 Next we claim that there is a unique  $E$-linear isomorphism  $\beta_i: W^i_0\lra U^{i+1}_0$
 such that $$\sL^{n+1-2i}\alpha _i  (x)=(n+1-2i)\epsilon  \sL ^{n-2i}\beta_i (x), \qquad \forall x\in W^i_0.$$
This equation has a unique solution $\beta_i (x)\in U^{i+1}$, because the left hand side belongs to $\epsilon  U^{n+1-i}=\epsilon \sL ^{n+2-2(i+1)}U^{i+1}$.
Since $\alpha _i (x)\in U^i_0$, $\sL^{n+2-2i}\alpha _i (x)=0$, thus $\beta_i(x)\in U_0^{i+1}$. It follows that  $\beta_i$ is a linear map $W^i_0\lra U^{i+1}_0$. 
We claim that $\beta_i$ is bijective. If $\beta_i(x)=0$, then $\alpha_i(x)=0$ since $\sL ^{n+1-2i}|V^i$ is bijective. Thus $x=0$.
For surjectivity, let $y\in U_0^{i+1}$, then there is a $z\in V^i$ such that $\sL^{n+1-2i}z=\epsilon \sL ^{n-2i}y$. 
Then $z$ is primitive since $\sL^{n+2-2i}z=\epsilon \sL^{n+1-2i}y=0$. Thus $\eta z=\sL u+x $ with $u\in W_0^{i-1}$, $x\in W_0^i$. 
Since $\sL^{n+1-2i}\eta z=\eta\epsilon \sL ^{n-2i}y=0$, we see that $u=0$. Thus $z=\alpha _i x$ and $y=\beta_ix$.

Now we define $\beta$ from $\beta_i$ using the commutativity with $\sL$, and define $\alpha$ from $\alpha_i$ and $\beta$ using  equation 
$$\alpha \sL^j(x) =\sL^j\alpha_i(x)  -j\epsilon \sL^{j-1}\beta_i(x), \qquad x\in W^i_0, \qquad j\le n-2i.$$
To check this is well defined, we notice that $\sL^j (x)=0$ implies either $x=0$ or $j=n-2i$.
It is clear that $\alpha$ and $\beta$ satisfy the condition of the Lemma. 
\end{proof}

 \begin{lem}
 For the lifting $\alpha$ in Lemma \ref{lem-ab}, the  two equations in Proposition \ref{prop-Lam-dec} hold. In particular, $\alpha$ is $\sLambda$ linear.
 \end{lem}
 \begin{proof}
 It is clear that the first equation holds by construction. For the second equation, let  $\sLambda'$ denote 
 the right hand side of the second equation in Proposition \ref{prop-Lam-dec}
$$ \sLambda' (\epsilon x+\alpha   y)=\epsilon \sLambda x+\alpha   (\beta^{-1} x+\sLambda y).$$
 We need only show that $\sLambda'$  satisfies the same equation as $\sLambda$ on $V^*$:
 $$[\sLambda', \sL]|_{V^i}=n+1-2i.$$
 This can be checked easily.
 \end{proof}
 
 \begin{lem} The lifting $\alpha$ in  Lemma \ref{lem-ab}  is the unique $\sLambda$-equivariant lifting. 
 \end{lem} 
\begin{proof} If there is  another $\sLambda$-equivariant lifting $\alpha'$,  then the difference is given by 
$$\alpha '-\alpha=\epsilon \delta$$
with  $\delta:  W^*\lra  U^*$ a linear map such that $\epsilon  \delta$  is equivariant with $\sLambda$:
$$\epsilon  \delta  \sLambda =\sLambda \epsilon  \delta=\epsilon  \sLambda \delta +\alpha  \beta^{-1}\delta.$$
Thus $\epsilon\beta ^{-1}\delta\in \Im \alpha \cap \Im \epsilon =0$ so is $\delta =0$.
\end{proof}

 \subsection{Symmetric pairings}\label{ss-sp}
 Let  notation be as in Theorem \ref{thm-fil-dec}. In the following, 
  we assume that $V^*$ has a symmetric,  non-degenerate intersection pairing:
 $$(\cdot, \cdot): V^*\otimes V^{n+1-*}\lra E, \qquad $$
 such that $F^2V^*$ and $F^1V^*$ are orthogonal complements to each other, and that the pairing is $\sL$-adjoint in the sense 
$$ (\sL x, y)=(x, \sL y),\qquad \forall x\in V^{*}, y\in V^{n-*}.$$
 This induces a non-degenerate $\sL$-adjoint pairing as follows:
 $$(\cdot, \cdot)_{i, 2-i}: G^iV^*\otimes G^{2-i}V^{n+1-*}\lra E.$$
Using the isomorphism $\beta: G^0V^*\lra G^2V^{*+1}$, we have the following  a  non-degenerate $\sL$-adjoint pairings $(\cdot, \cdot)_{i,i}$ on $G^iV^*$ for $i=0,  2$:
 $$(\cdot, \cdot)_{i, i}: G^iV^*\otimes G^iV^{n+i-*}\lra E,$$
 $$(x, y)_{0, 0}:=(x, \beta y)_{0, 2}, \qquad\forall  (x, y)\in G^0V^*\times G^0V^{n-*},$$
 $$ (u, v)_{2, 2}:=(\beta^{-1} u, v)_{0,2},\qquad\forall (u, v)\in G^2V^*\times G^2V^{n+2-*}.$$
 \begin{prop}\label{prop-sym}
 With the assumption and notation as above, we have the following assertions:
 \begin{enumerate}
 \item For each $i$, the pairing $(\cdot, \cdot)_{i, i}$ is symmetric, non-degenerate, and $\sL$-adjoint. 
 \item The $\Im \alpha ^0$ is perpendicular to $\Im\alpha ^1+\Im \alpha ^0$. 
 \end{enumerate}
 Thus the pairing $(\cdot,\cdot)$ is transformed by $\alpha$ to the following pairing on $\bigoplus G^iV^*$:
 $$((x, y, z), (x', y', z'))=(x, z')_{0, 2}+(y, y')_{1,1}+(z, x')_{2, 0} $$
  $$\forall (x, x')\in G^0V^*\times G^0V^{n+1-*}, \quad \forall (y, y')\in G^1V^*\times G^1V^{n+1-*},
 \quad \forall (z, z')\in G^2V^*\times G^2V^{n+1-*}.$$ 
  \end{prop}
 
 \begin{proof} It is easy to see that $(\cdot, \cdot)_{1,1}$ is symmetric, non-degenerate, and $\sL$-adjoint,
 and that  the isomorphism $\alpha$ in Theorem \ref{thm-fil-dec} restricting on $G^1V^*$ is an isometry. 
 Thus the quotient $V^*/\alpha G^1V^*$ is isomorphic to the orthogonal complement of $\Im\alpha _1$ with induced two-step filtration.
  So we have reduced the proof to the case $G^1V^*=0$.

 Let  $\CD(V)^*$ denote the dual Lefschetz structure  for $V^*$ with 
 $$\CD (V)^i=\Hom (V^{n+1-i}, E), \qquad \sL \psi=\psi\circ \sL, \quad \forall \psi\in \CD (V)^*$$ 
 Then the pairings $(\cdot, \cdot)$ and $(\cdot, \cdot)_{2, 0}$ induce   isomorphisms of Lefschetz modules 
 $$V^*\lra \CD (V)^*, \qquad G^0V^*\lra \CD (F^2V)^*.$$
 These isomorphisms must be $\sLambda$-equivariant. 
 
 Since $F^2V^*$ is isotropic in $V^*$, we have the following equality:
 $$(\alpha^0 x+\alpha^2  z, \  \alpha ^0 x'+\alpha ^2 z')=(x, \beta^{-1}z')_{0, 0}+(x', \beta^{-1} z)_{0, 0}+(\alpha ^2 z, \alpha^2  z'),$$
 $$\forall (x, x')\in G^0V^*\times G^0V^{n+1-*}V, \, \qquad \forall (z, z')\in G^2V^*\times G^2V^{n+1-*}V^*.$$
 Now  apply the adjoint property of $\sLambda$ to obtain the following equality:
 $$(\sLambda (\alpha ^0x+\alpha ^2 z), \ \alpha ^0 x'+\alpha ^2 z')=(\alpha ^0x+\alpha ^2 z,\  \sLambda (\alpha ^0 x'+\alpha ^2 z')).$$
 By Theorem \ref{thm-fil-dec}, this means the following equality:
 $$(\alpha ^0\sLambda x +\alpha ^2\sLambda z+\alpha ^0\beta ^{-1}z, \ \alpha ^0 x'+\alpha ^2 z')
 =(\alpha ^0 x+\alpha ^2 z, \ \alpha ^0\sLambda x'+\alpha ^2\sLambda z'+\alpha ^0\beta ^{-1} z').$$
  We  apply this equation to each of the following  cases
 \begin{enumerate}
 \item  $x=x'=0$: $(z', z)_{2, 2}=(z, z')_{2, 2}$;
 \item $z=x'=0$: $(\sLambda x, z')_{0, 2}=(x, \sLambda z')_{0, 2} +(\alpha ^0 x,\alpha ^0\beta ^{-1}z') $;
  \item $z=z'=0$: 
 $(\alpha^0 \sLambda x, \alpha^0  x')=(\alpha^0  x, \alpha^0  \sLambda x')$.
 \end{enumerate}
 The first case shows that $(\cdot, \cdot)_{2,2}$ is symmetric, and so is $(\cdot, \cdot)_{0, 0}$ as 
  $$(x, y)_{0, 0}=(x, \beta y)_{0, 2}=(\beta x, \beta y)_{2,2}, \qquad \forall x\in G^0V^*, y\in G^0V^{n-*}.$$
The second case then shows that $(\alpha ^0 z, \alpha ^0 x')=0$ for all $(x, x')\in G^0V^*\times G^0V^{n+1-*}$.
 Thus $\Im \alpha ^0$ is isotropic. 
 The third case is trivial. 
 The rest of the proposition is straightforward.
 \end{proof}

  \subsection{Standard conjectures}
  
   In this section, we work on   $E=\BR$.
  \begin{defn}[Standard conjecture]\label{def-conj}
   Let $M$ be a $\BR[\sL]$-module with a symmetric 
    pairing with center $n/2$:
    $$M^*\times M^{n-*}\lra \BR$$
    such that $\sL$ is self-adjoint  in the sense 
$$ (\sL x, y)=(x, \sL y),\qquad \forall x, y\in V^{*}.$$
We say that the {\em the standard conjecture} holds for  $(M, (\cdot, \cdot))$ with   center $n/2$,
   if the following two conditions hold:
   \begin{enumerate}
   \item {\em Hard Lefschetz theorem}: $M$ is a Lefschetz module with center $n/2$;
   \item {\em Hodge index theorem}: the pairing $(-1)^i(\cdot, \cdot)$ is positive definite on the primitive component $M_0^i=\Ker (\sL ^{n+1-2i}|M^i)$.
   \end{enumerate}
   \end{defn}

 \begin{prop}\label{prop-pos} With setting as in  Theorem \ref{thm-fil-dec}, the Hodge index theorem for $(V^*, (\cdot , \cdot )_V)$ is equivalent to the Hodge index theorem for 
 $(G^1V^*, (\cdot, \cdot)_{1,1})$ and 
 $(G^0V^*, (\cdot, \cdot )_{0, 0})$ 
 \end{prop}
 \begin{proof} By Proposition \ref{prop-sym}, it is easy to reduce the proof  to the case $G^1V^*=0$. So let's assume $G^1V^*=0$ 
  and  let $\alpha ^0 x+\alpha ^2z$ be a primitive element in $V^i$.
 Then $x\in G^0V^i, z\in G^2V^i$ with $i\le (n+1)/2$ and 
 $$0=\sL ^{n+2-2i}(\alpha ^0 x+\alpha ^2 z)=\alpha ^0 \sL ^{n+2-2i} x+\alpha ^2 (\sL^{n+2-2i}z+(n+2-2i)\sL ^{n+1-2i}\beta x).$$
 This implies  $x=x_i+\sL x_{i-1}$ with $x_i, x_{i-1}$ primitives, and that $z=-(n+2-2i)\beta x_{i-1}$, where $x_i=0$ if $i=(n+1)/2$. 
 Now 
 \begin{align*}
 \sL ^{n+1-2i}(\alpha ^0  x+\alpha ^2 z)=&\alpha ^0 \sL ^{n+1-2i} x+
 \alpha ^2 (\sL^{n+1-2i}z+(n+1-2i)\sL ^{n-2i}\beta x)\\
 =&\alpha ^0  \sL ^{n+2-2i}x_{i-1}+\alpha ^2 (-\sL^{n+1-2i}\beta x_{i-1}+(n+1-2i) \sL ^{n-2i} \beta x_i).
 \end{align*}
 
 It follows that 
 \begin{align*}
 &(\alpha ^0 x+\alpha ^2 z, \sL ^{n+1-2i}(\alpha ^0 x+\alpha ^2 z))\\
 =&(\beta ^{-1}z, \sL ^{n+2-2i}x_{i-1})_{0, 0}+(-\sL^{n+1-2i}x_{i-1}+(n+1-2i) \sL ^{n-2i} x_i, x)_{0, 0}\\
 =&-(n+3-2i)(x_{i-1}, \sL ^{n+2-2i}x_{i-1})_{0, 0}+(n+1-2i) (x_i, \sL ^{n-2i} x_i)_{0, 0}.
 \end{align*}
The positivity of the last quantity is equivalent to the Hodge index theorem for $G^0V^*$. 
\end{proof}

In the rest of this section, we 
 let $V^*\in \BM$ be a $\sL$-module over $\BR$  with following structures:
 \begin{enumerate}
 \item A three-step  filtration for $V^*$ in objects in $\BM$:
$$0\emb F^2V^*\emb F^1 V^*\emb F^0V^*=V^*$$
with graded quotients  $G^iV:=V^i/V^{i+1}$. Assume that $G^2V^*$ is finite-dimensional. 
\item  a symmetric intersection pairing on $V^*$:
 $$(\cdot, \cdot): V^*\otimes V^{n+1-*}\lra E, \qquad $$
 such that $F^2V^*$ and $F^1V^*$ are orthogonal complements to each other, and that the pairing is $\sL$-adjoint in the sense 
$$ (\sL x, y)=(x, \sL y),\qquad \forall x\in V^{*}, y\in V^{n-*}.$$
\item an  operator $\epsilon\in\Hom _\BM (V^*, V^{*+1})$ such that $\Im (\epsilon)\subset F^2V^*$, $\Ker(\epsilon)\supset F^1V^*$, and that 
$$ (\epsilon x, y)=(x, \epsilon y),\qquad \forall x\in V^{*}, y\in V^{n-*}.$$
The induced map $G^0V^*\lra F^2V^{*+1}$ is still denoted by $\epsilon$.
\end{enumerate}
Notice that at this point, we do not assume that $G^iV$ are Lefschetz modules and that $(\cdot, \cdot)$ is non-degenerate
on $V^*$. These structures allows us to define the following objects:
\begin{enumerate}
\item New operators $\sL(c)=\sL+c\epsilon$ for $c\in \BR$.
\item A symmetric pairing $(\cdot, \cdot)_0$ on $G^0V^*$ by
$$(x, y)_0=(x, \epsilon y).$$
\item A symmetric pairing $(\cdot, \cdot)_1$ on $G^1V^*$:
$$(x, y)_1=(\bar x, \bar y), \qquad x\in G^1V^*, y\in G^1V^{n+1-*}$$
where  $\bar x\in F^1V^*, \bar y\in F^1V^{n+1-*}$ are any liftings of $x, y$.
\end{enumerate}

The  main result in this section  is the following:

\begin{thm}\label{thm-vg1}
 Assume the standard conjecture for $G^0V^*$.
  Then the  following statements hold:
 \begin{enumerate}
 \item 
 The  standard conjecture for $V^*$  implies  the standard conjecture for $G^1V^*$.
  \item The standard conjecture for $G^1V^*$ implies the standard conjecture for $V^*$ for any Lefschetz operator  of the form
 $\sL (c)=\sL+c\epsilon$ with   $c\in \BR$ sufficiently large.
 \end{enumerate}
 \end{thm}

Let us start with part 1 of Theorem \ref{thm-vg1}.

\begin{lem}\label{lem-fiac}
Assume  the  standard conjectures for $G^0V^*$ and $V^*$. Then $G^1V^*$ satisfies the standard conjecture. More precisely, the following statements hold:
\begin{enumerate}
\item $\epsilon: G^0V^*\iso F^2V^{*+1}$ as an $\sL$-modules. 
\item The submodule $C^*:=F^2V^*+\sLambda F^2V^{*+1}$ is stable under $\sL$ and $\sLambda$.
\item The $C^*$ satisfies the standard conjecture with center $(n+1)/2$, and the projection 
to $G^2V^*$ induces an isomorphism $\sLambda F^2V^*\iso G^0V^*$.
\item Let $D^*$ denote the orthogonal complement of $C^*$ in $V^*$. 
Then  $$F^1V^*=F^2V^*+D^*, \qquad G^1V^*\iso D^*.$$
\item The standard conjecture holds for $G^1V^*\iso D^*$.
\end{enumerate}
\end{lem}
\begin{proof}
By the standard conjecture for $V^*$, the intersection pairing on $V^*$ is non-degenerate.
Since $F^2V^*$ is perpendicular to $F^1V^*$ in $V^*$, the intersection pairing induces maps:
$$F^2V^*\emb  (G^0V^{n+1-*})^\vee\iso G^0V^{*-1},$$
where the last isomorphism follows from the standard conjecture for $G^0V^*$.
Combining this with the  map $\epsilon: G^0V^{*-1}\lra F^2V^*$ we obtain a chain of maps:
$$G^0V^{*-1}\overset \epsilon\lra F^2V^*\emb G^0V^{*-1}.$$
The composition is the identity map. Thus all these maps are bijective. 
This proves part 1.

For part 2, the stability of $C^*$  under $\sL$ is easy as 
$\sL\sLambda=[\sL,\sLambda]+\sLambda \sL$.
For $\sLambda$, we need to show that $\sLambda^2F^2V^*$ is included into $C^*$.
Since by part 1, every element in $F^2V^*$ can be written as a linear combination of elements of the form $\sL^i \epsilon x$ with $x$ a primitive element in $G^0V^*$, thus we need only show that $\sLambda ^2\sL^i \epsilon x\in C^*$ for al;l $x\in G^0V^*_0$.
It is easy to see that 
$$\sLambda \sL^i\epsilon x=\sL^i\sLambda \epsilon x\pmod {F^2V^*}, 
\qquad \sLambda^2\sL^i \epsilon x \equiv \sL^i\sLambda ^2\epsilon x\pmod{C^*}.$$
Thus it suffices to show $\sLambda ^2\epsilon x=0$.
Set $j=\deg x\le n/2$. Then 
$$\sL ^{n+1-2j}\epsilon x=\epsilon \sL^{n+1-2j} x=0.$$
Since $\deg \epsilon x=j+1$, this follows that $\epsilon x=\sL x_1+x_2$
with $x_1, x_2$ primitive in $V^*$. 
This shows that $\sLambda ^2\epsilon x=0$.
This proves part 2.

By part 2,  we see that $C^*$ is a Lefschetz submodule of $V^*$ with center $(n+1)/2$.
The standard conjecture for $V^*$  applies to $C^*$. Thus the induced pairing on $C^*$ is perfect. 
Since the $F^2V^*$ is isotropic under the intersection pairing, the pairing 
$$F^2V^*\times \sLambda F^2V^*\lra \BR$$
is non-degenerate on $F^2V^*$; and thus  it is perfect as $\dim \sLambda F^2V^*\le \dim F^2V^*$.
It follows that the projection $C^*\lra G^2V^*$ induces a bijection $\sLambda F^2V^*\iso G^0 V^*$.
This proves part 3.

By part 3,  $D^*$ is a Lefschetz module 
with center $(n+1)/2$ and satisfies the Hodge index theorem. As the $F^2V^*$ is the orthogonal complement of $F^1V^*$,
we see that $F^1V^*$ is the orthogonal complement of $F^2V^*$. Thus  the following identities hold:
$$F^1V^*=F^2V^*+D^*, \qquad G^1V^*=D^*.$$
Thus we have shown that $G^1V^*$ satisfies the standard conjecture. 
\end{proof}

Now we want to prove the second part of Theorem \ref{thm-vg1}.  Notice that the action of $\sL(c)$ is same as that of $\sL$ on $F^1V^*$.
By Proposition \ref{prop-pos}, we need only prove the following:

\begin{lem} With assumption as in Theorem \ref{thm-vg1}, part 2, then there is a $c\in\BR_{\ge 0}$   with the following properties:
\begin{enumerate}
\item $V^*$ is a Lefschetz module with center $(n+1)/2$ for $\sL (c)$;
\item $G^1V^*$ satisfies the Hodge index theorem. 
\end{enumerate}
\end{lem}

\begin{proof}
 The assumption on the non-degeneracy of  $F^2V^*\times G^0V^*$ implies that $F^2V^*$ is a Lefschetz module with center $(n+1)/2$.
Combining with a standard conjecture for $G^1V^*$, we also have the non-degeneracy of the pairing on $V^*$.

Now we apply Proposition \ref{prop-lef-dec} to the exact sequence: 
$$0\lra F^2V^*\lra F^1V^*\lra G^1V^*\lra 0$$
we obtain a splitting $\alpha: G^1V^*\lra F^1V^*$ of $\sL$ modules. 
Let $C^*$ denote the orthogonal complement of $\Im (\alpha)$ in $V^*$.
Then  $C^*$ is an $\sL(c)$ module with a non-degenerate  intersection and sits  in an exact sequence 
$$0\lra F^2V^*\lra C^*\lra G^0V^*\lra 0.$$
We want to show that $C^*$ is a $\sL(c)$ module with center $(n+1)/2$.

For any $i< (n+1)/2$, consider the map
$$\sL(c)^{n+1-2i}: C^i\lra C^{n+1-i} .$$
It is easy to see that 
$$\sL(c)^{n+1-2i}=\sL^{n+1-i}+(n+1-i)c\sL^{n-2i}\epsilon.$$
Since $C^i$ and $C^{n+1-i}$ have the same dimension say $d_i$, the set $S_i$ of $c\in \BR$ so that $\sL (c)^{n+1-2i}$ 
is not injective is the set of roots of a polynomial equation $P_i=0$ of degree $\le d_i$.
We need only   show that $P_i\ne 0$. 
Notice that $P_i=0$ is equivalent to that two operators $\sL^{n+1-2i}$ and $\sL^{n-2i}\epsilon$ have a common  null vector $x\in C^i$.
Then $\epsilon x\in G^2V^{i+1}$. Since $F^2V^*$ is a Lefschetz module with center $(n+2)/2$, the equation $\sL^{n-2i} \epsilon x=0$
implies that  $\epsilon x=0$.
Thus $x\in G^2V^i$. Then the equation $\sL^{n+1-2i} x=0$ implies that $x=0$.
This finishes the proof that $V^*$ is a Lefschetz module with center $(n+1)/2$, and thus the first part of Lemma. 

For the second part of Lemma, we notice that the lifting $\alpha ^0$ for $\sL (c)$ does not depend on the choice of $c$ because of   
explicit formulae for $\sL$ and $\sLambda$ in Theorem \ref{thm-fil-dec}. Thus we can use the same formula to obtain 
$$\beta (c)=\beta +c\epsilon: G^0V^*\iso F^2V^{*+1}.$$
It follows that the induced pairing $(\cdot, \cdot )_{0, 0}(c)$ on $G^0V^*$ has the form
$$(x, y)_{0, 0}(c)=(x, y)_{0, 0}+c (x, y)_0.$$
Now for an $i\le n/2$, and the pairing on the primitive part $G^0V^i_0$ multiplied by $(-1)^i$  has the form
$$(-1)^i(x, y)_{0, 0}(c)=(-1)^i(x, y)_{0, 0}+c(-1)^i (x, y)_0.$$
Since $(-1)^i (\cdot, \cdot )_0$ is positive definite, the above pairing is positive definite for $c$ sufficiently large. 
\end{proof}

\end{document}